\documentclass{amsart} 
\usepackage{amsmath,amsfonts,amsthm,amssymb,mathrsfs,latexsym,paralist, xypic}
\usepackage{enumitem}   
\usepackage{upref}      
\usepackage{stmaryrd}   
\usepackage[normalem]{ulem} 
\usepackage{tikz}
	\usetikzlibrary{arrows,automata}
	\tikzstyle{V}=[fill=black,circle,scale=0.4, outer sep = 4pt]
\usepackage[
]{hyperref}

\newtheorem{thm}{Theorem}[section]
\newtheorem{prop}[thm]{Proposition}
\newtheorem{cor}[thm]{Corollary}
\newtheorem{lemma}[thm]{Lemma}

\theoremstyle{remark}
\newtheorem{rmk}[thm]{Remark}

\theoremstyle{definition}
\newtheorem{defn}[thm]{Definition}

\DeclareMathOperator{\supp}{supp}

\DeclareMathOperator{\dom}{dom}

\newcommand{\z}{^{(0)}}
\renewcommand{\2}{^{(2)}}
\newcommand{\inv}{^{-1}}

\newcommand{\bi}{\begin{itemize}}
\newcommand{\ei}{\end{itemize}}
\newcommand{\be}{\begin{enumerate}}
\newcommand{\ee}{\end{enumerate}}

\newcommand{\CC}{\mathbb{C}}

\newcommand{\Iso}{\operatorname{Iso}}
\newcommand{\norm}[1]{\left\|#1\right\|}  
\newcommand{\abs}[1]{\left\vert #1 \right\vert}

\date{\today}

\let\mbb\mathbb
\let\mf\mathfrak
\let\mc\mathcal
\let\sset\subseteq

\newcommand{\gpd}{\mc{G}} 
\newcommand{\agpd}{\mc{S}} 
\newcommand{\qgpd}{\mc{Q}}	
\newcommand{\specB}{\mf{B}} 
\newcommand{\qsec}{\mf{s}}	
\newcommand{\WeylGpdElement}{\Gamma}

\begin{document}

\title{Analyzing the Weyl construction for dynamical Cartan subalgebras}

\author{A.\ Duwenig}
\address[A.\ Duwenig]{School of Mathematics and Applied Statistics, University of Wollongong, Wollongong, Australia}
\email{aduwenig@uow.edu.au}

\author{E.\ Gillaspy}
\address[E.\ Gillaspy]{Department of Mathematical Sciences, University of Montana, 32 Campus Drive \#0864, Missoula, MT 59812, USA}
\email{elizabeth.gillaspy@mso.umt.edu}

\author{R.\ Norton}
\address[R.\ Norton]{Mathematics Department, Fitchburg State University, 160 Pearl Street,
Fitchburg, MA 01420, USA}
\email{rnorton9@fitchburgstate.edu}

\maketitle
\begin{abstract}
    
When the reduced twisted $C^*$-algebra $C^*_r(\gpd, c)$ of a non-principal groupoid $\gpd$ admits a Cartan subalgebra, Renault's work on Cartan subalgebras implies the existence of another groupoid description of $C^*_r(\gpd, c)$.  In an earlier paper, joint  with Reznikoff and Wright, we identified situations where such a Cartan subalgebra arises from a subgroupoid $\agpd$ of $\gpd$.
In this paper, we study the relationship between the original groupoids $\agpd, \gpd$ and the Weyl groupoid and twist associated to the Cartan pair. We first identify the spectrum $\specB$ of the Cartan subalgebra $C^*_r(\agpd, c)$. We then show that the quotient groupoid $\gpd/\agpd$ acts on $\specB$, and that the corresponding action groupoid is exactly the Weyl groupoid of the Cartan pair. Lastly we show that, if the quotient map $\gpd\to\gpd/\agpd$ admits a continuous section, then the Weyl twist is also given by an explicit continuous $2$-cocycle on $\gpd/\agpd \ltimes \specB$.
\end{abstract}

\setcounter{tocdepth}{1} 

~

\tableofcontents

\pagebreak[3]\section{Introduction}

One of the earliest theorems about $C^{*}$-algebras, the Gelfand--Naimark Theorem, establishes that any commutative $C^{*}$-algebra is of the form $C_0(X)$ for a locally compact Hausdorff space $X$. 
In addition to inspiring the ``noncommutative topology'' approach to $C^{*}$-algebras, the Gelfand--Naimark Theorem  has also led researchers to search for Abelian subalgebras inside noncommutative $C^{*}$-algebras, with the goal of using topological tools to analyze the Abelian
subalgebra, and from there
to obtain a better understanding of its noncommutative host.
This program has been particularly successful when the subalgebra
is
{\em Cartan} (see Definition \ref{def:cartan} below);
it has
enabled progress on
Elliott's classification program for $C^{*}$-algebras \cite{barlak-li, li-cartan-existence} as well as the theory of  continuous orbit equivalence of topological dynamical systems \cite{matsumoto-matui, brownlowe-carlsen-whittaker}.
Even beyond the setting of Cartan subalgebras,
many authors (cf.~\cite{ brownlowe-carlsen-whittaker,BFPR:graded-and-twisted,IKRSW:2020}) have successfully extended structural information from more general Abelian subalgebras to the containing $C^{*}$-algebras.

In this paper, we focus our attention on certain Cartan subalgebras which appear in a rather unexpected context.
 Renault proved in \cite{renault-cartan}, building on earlier work of Kumjian \cite{Kumjian:diag}, that if a $C^{*}$-algebra $A$
admits a Cartan subalgebra, then $A$ 
is isomorphic to the reduced $C^{*}$-algebra $C^{*}_r( G, \Sigma)$ of a twist $\Sigma$ over
a groupoid ${G}$,  and the Cartan subalgebra is realized as $C_0({G}\z).$ 
The  groupoids $G$ appearing in Renault's analysis must satisfy a number of structural contraints;  for example, they are always
topologically principal.
If ${G}$ is not topologically principal, then $C_0({G}\z)$ is not a Cartan subalgebra inside $C^{*}_r({G} , \Sigma )$ for any twist $\Sigma$ over ${G} $.
Nevertheless, there are many such groupoids whose twisted $C^{*}$-algebras contain Cartan subalgebras.
Examples include the rotation algebras $A_\theta \cong C^{*}_r(\mbb{Z}^2, c_\theta)$ and the $C^{*}$-algebras of directed graphs which do not satisfy Condition (L).

Together with Reznikoff and Wright, in \cite[Theorem 3.1]{DGNRW:Cartan} we identified a large family of twisted groupoid $C^{*}$-algebras,
associated to non-principal groupoids $\mc G$, which contain Cartan subalgebras. Moreover, these Cartan subalgebras are evident at the level of the groupoid $\mc G$: they arise from a subgroupoid $\mc S$ of $\mc G$.
The existence of a Cartan subalgebra in $C^{*}_r(\mc{G} , c)$ implies, by \cite{renault-cartan}, the existence of a topologically principal groupoid $G$, the so-called {\em Weyl groupoid,}
and a twist $\Sigma$ over $G$ such that $C^{*}_r(\mc{G} , c) \cong C^{*}_r(G, \Sigma ).$ If $\mc{G} $ is a discrete group, and $\mc{S} \leq \mc{G} $ satisfies the hypotheses of \cite[Theorem 3.1]{DGNRW:Cartan} so that $C^{*}_r(\mc{S}, c)$ is a Cartan subalgebra of $C^{*}_r(\mc{G} , c)$, then \cite[Theorem 5.2]{DGNRW:Cartan} establishes that 
\[ G =  (\mc{G} /\mc{S}) \ltimes \widehat{\mc{S}}\]
as long as the action of $\mc{G} /\mc{S}$ on $\widehat{\mc{S}}$ is topologically free.  Moreover, \cite[Theorem 5.8]{DGNRW:Cartan} shows that in this case, the twist $\Sigma$ arises from a 2-cocycle 
on $G$, which we described explicitly in \cite[Lemma 5.6]{DGNRW:Cartan}. 

The preprint \cite{IKRSW:2020} came to our attention as we were finalizing \cite{DGNRW:Cartan}, and we were struck by the structural
parallels between the two papers' main results.  
In  \cite[Theorem 3.4]{IKRSW:2020}, the authors show that if a subgroupoid $\mc{A} $  of a groupoid $\mc{G} $ consists of a closed normal bundle of Abelian groups, then 
\[C^{*}_r(\mc{G}) \cong C^{*}_r( \widehat{\mc{A} } \rtimes (\mc{G}/\mc{A}) , \Sigma ).\]
However, they only establish that $C^{*}_r(\mc{A} )$ is Cartan in $C^{*}_r(\mc{G} )$ 
if $\mc{G} $ is {\'e}tale and $\mc{G} /\mc{A} $ is topologically principal \cite[Theorem 5.6]{IKRSW:2020}, 
and they do not analyze the structure of twisted groupoid $C^{*}$-algebras $C^{*}_r(\mc G, c).$
As non-trivial discrete groups are 
never topologically principal, this  excludes the setting of \cite[Theorem 5.8]{DGNRW:Cartan}. Moreover,  the formula given in \cite{IKRSW:2020} for the twist $\Sigma $ is not 
explicit.  In particular, it is unclear when, or whether, the twist $\Sigma$ can be realized via a 2-cocycle on the groupoid $\widehat{\mc{A} } \rtimes (\mc{G} /\mc{A} ).$

In this paper, we bridge the gap between \cite{DGNRW:Cartan} and \cite{IKRSW:2020}.  Our first main result, Theorem \ref{thm:varphi}, establishes 
that when a subgroupoid $\agpd$ of an {\'e}tale groupoid $\gpd$ satisfies the hypotheses of \cite[Theorem 3.1]{DGNRW:Cartan}, so
 that $C^{*}_r(\agpd, c)$ is Cartan in
 $C^{*}_r(\gpd , c)$, then the Weyl groupoid associated to the Cartan pair
$ \left( C^{*}_r(\gpd , c), C^{*}_r(\agpd , c)\right)$ is an action groupoid
\[  (\gpd /\agpd)  \ltimes \specB,\]
where $\specB$ denotes the spectrum of the commutative algebra $C^{*}_r(\agpd , c)$.  When the 2-cocycle $c$ is trivial,    
$\specB $ agrees with the space $ \widehat{\agpd }$ of  \cite{IKRSW:2020}, and (translating our left action of $\gpd /\agpd $ into a right action) we see that our groupoid $(\gpd/\agpd) \ltimes \specB$ agrees with the groupoid $\widehat{\agpd } \rtimes (\gpd /\agpd) $ of \cite{IKRSW:2020}; see Remark \ref{rmk:comparison-with-IKRSW}. 
Theorem \ref{thm:varphi} is a substantial improvement over \cite[Theorem 5.2]{DGNRW:Cartan}.  Not only do we extend \cite[Theorem 5.2]{DGNRW:Cartan} from  groups to groupoids, we
also
show that the hypothesis of topological freeness in the latter
theorem is always satisfied.

Our second main result, Theorem \ref{thm:psi}, shows 
in particular
that if the quotient map $q\colon \gpd \to \gpd/\agpd$
admits a continuous section $\qsec$, then the Weyl twist $\Sigma$ of the Cartan pair $\left( C^{*}_r(\gpd , c), C^{*}_r(\agpd , c)\right)$  is given by a continuous 2-cocycle $C^\qsec$ on $(\gpd/\agpd) \ltimes \specB$.  The formula for $C^\qsec$ is given in Proposition \ref{prop:def:cocyle-for-Weyl}.
  Our precise description 
  of the twist $\Sigma$ represents both an extension of \cite[Theorem 5.8]{DGNRW:Cartan} to a broader setting 
  and an improvement on \cite[Theorem 3.4]{IKRSW:2020} in the
  {\'e}tale case.

A detailed analysis of the Weyl groupoid  underlies our results in this paper, 
and we expect that some of the technical results we have obtained, such as Proposition \ref{prop:Weyl-gpd-similar-to-twist} and Corollary \ref{cor:specB-is-spectrum}, will also be of general interest.
Proposition \ref{prop:Weyl-gpd-similar-to-twist} 
gives a description of the equivalence relation underlying the Weyl groupoid which, to our knowledge, has not appeared before in the literature.
Corollary~\ref{cor:specB-is-spectrum} 
describes the spectrum of the twisted groupoid $C^{*}$-algebra $C^{*}_r(\agpd , c)$ if $\agpd $ is a  bundle of Abelian groups 
and $c$ is symmetric on $\agpd$.
This description may
be known to experts -- indeed, it is similar to results such as \cite[Corollary 3.4]{MRW96}, \cite[Proposition 5]{Goehle:gp-bdl-duality}, and  \cite[Remark 5.2]{crytser-nagy} -- but we were unable to locate a reference  in the literature for the precise result we needed.

This paper is organized as follows.  We recall the relevant definitions of groupoids and Cartan subalgebras in Section \ref{sec:prelim}, and we
establish Proposition \ref{prop:Weyl-gpd-similar-to-twist}.  The first step in providing an explicit description of the Weyl groupoid associated to a Cartan pair $(A, B)$ is identifying the topological space $\widehat B$;  Section \ref{sec:abelian-gp-bdl} is devoted to describing $\widehat B$ in the case when $B = C^{*}_r(\agpd , c)$ arises from a bundle of discrete Abelian groups.

Section \ref{sec:weyl-paper2} contains our first main result.
We show in Theorem \ref{thm:varphi} that if $(A, B) = ( C^{*}_r(\gpd , c), C^{*}_r(\agpd, c))$ is one of the Cartan pairs identified by \cite[Theorem 3.1]{DGNRW:Cartan}, then its Weyl groupoid is   an action groupoid, arising from an action of the quotient groupoid $\gpd /\agpd$ on the spectrum $\widehat{B}\cong\specB$.
Given a section $\qsec \colon \gpd /\agpd \to \gpd $ of the quotient map, we describe in Section \ref{sec:cocycle} a 2-cocycle $C^{\qsec}$ on the action groupoid $(\gpd /\agpd )\ltimes \specB.$  When $\qsec$ is continuous, we prove in Theorem \ref{thm:psi} that $C^{\qsec}$ gives the Weyl twist over $(\gpd /\agpd) \ltimes \specB$, so that  $C^{*}_r(\gpd , c) \cong C^{*}_r((\gpd / \agpd) \ltimes \specB, C^{\qsec})$.

\subsection*{Acknowledgments}

The research presented here grew out of the work \cite{DGNRW:Cartan} which was begun during the authors' participation in the 2018 BIRS workshop ``Women in Operator Algebras'' and continued during a 2019 stay at the Mathematical Sciences Research Institute.  The latter stay was supported by the National Security Agency under Grant No.\ H98230-19-1-0119, The Lyda Hill Foundation, The McGovern Foundation, and Microsoft Research.

The first-named author would further like to thank Aidan Sims and Chris Bruce for helpful discussions, and the Department of Mathematics and Statistics at the University of Victoria for its hospitality. 

\pagebreak[3]\section{Preliminaries on Cartan subalgebras and the Weyl construction}
\label{sec:prelim}

Intuitively, a {\em groupoid} $\gpd $ is a generalization of a group in which multiplication is only partially defined.  More precisely, a groupoid is a set $\gpd $, together with a subset $\gpd\2 \subseteq \gpd  \times \gpd $; a multiplication map $(\gamma, \eta) \mapsto \gamma \eta $ from $\gpd\2$ to $\gpd $; and an inversion map $\gamma \mapsto \gamma\inv$ from $\gpd  $ to $\gpd $, which behave like multiplication and inversion do in groups wherever they are defined.
The {\em unit space} of $\gpd $ is $\gpd\z = \{ \gamma \gamma\inv: \gamma \in \gpd \}$.  We then have  range and source maps $r, s: \gpd  \to \gpd\z$ given by 
\[ r(\gamma) = \gamma \gamma\inv, \quad s(\gamma) := \gamma\inv \gamma,\]
which satisfy $r(\gamma) \gamma = \gamma = \gamma s(\gamma)$ for all $\gamma \in \gpd $. 
Given $u \in \gpd\z$, we write $\gpd_u := \{ \gamma \in \gpd: s(\gamma)= u\}$ and $\gpd^u := \{ \gamma \in \gpd: r(\gamma) = u\}$.
We can also describe $\gpd\2$ using the range and source maps:
\[ \gpd\2 = \{ (\gamma, \eta) \in \gpd  \times \gpd : s(\gamma) = r(\eta)\}.\]

In this paper we will assume that $\gpd $ is equipped with a locally compact Hausdorff topology with respect to which the multiplication and inversion maps are continuous.
The groupoids considered in this paper will also be {\em {\'e}tale} -- that is, $r$ and $s$ will be local homeomorphisms.

The  link between groupoids and Cartan subalgebras was established in the seminal papers \cite{Kumjian:diag, renault-cartan}.

\begin{defn}[{\cite[Definition 5.1]{renault-cartan}}]
\label{def:cartan}
Let $A$ be a $C^{*}$-algebra.  A $C^{*}$-subalgebra $B$ of $A$ is a {\em Cartan subalgebra} if:
\begin{enumerate}[label=(\arabic*)]
\item\label{item:def-Cartan:masa} $B$ is a maximal Abelian subalgebra
of $A$.
\item\label{item:def-Cartan:CondExp} There exists a faithful conditional expectation $\Phi\colon  A \to B$.
\item\label{item:def-Cartan:normalizer} $B$ is regular; that is, the normalizer of $B$, 
\[ N(B) := \{ n \in A \text{ such that } nbn^{*}, n^{*}bn \in B \text{ for all } b \in B\},\]
generates $A$ as a $C^{*}$-algebra.
\item\label{item:def-Cartan:Approx1} $B$ contains an approximate identity for $A$.
\end{enumerate}
\end{defn}
For this section, we fix a Cartan subalgebra $B$ of some separable $C^*$-algebra $A$. 
Let us first recall how $(A,B)$ gives rise to a topologically principal groupoid and twist (cf.~\cite[1.6]{Kumjian:diag} or \cite[Proposition 4.7]{renault-cartan}), and then gather a few tools to study them. Let $\widehat{B}$ be the spectrum of $B$, viewed as the space of one-dimensional representations of $B$ (a subspace of $B^{*}$, the space of linear functionals on $B$), and let $\Omega\colon C_0 (\widehat{B})\to B$ be the Gelfand representation. 

As $B$ contains an approximate identity for $A$, if $n \in N(B)$ then $n^{*}n, nn^{*} \in B$. For each $n\in N(B)$, there exists a unique partial homeomorphism $\alpha_n$ with domain
    \[
        \mathrm{dom}(n)
        :=
        \left\{
            x\in \widehat{B}
            \,\middle|\,
            \Omega\inv(n^{*}n)(x) = x(n^{*}n)>0
        \right\}
    \] and with codomain $\mathrm{dom}(n^{*})$, that satisfies
    \begin{align}\label{eq:defining-eq-for-alpha}
        n^{*} \Omega(f) n 
        =
        \Omega(f\circ\alpha_{n}) \, n^{*} n
    \end{align}
    for all $ f \in C_0(\widehat{B}).$
    If $n,m\in N(B)$, then 
    $\alpha_n \circ \alpha_m = \alpha_{nm}$ and $\alpha_{n^{*}} = \alpha_n\inv$  (cf.~\cite[Corollary 1.7]{Kumjian:diag}).

    The {\em Weyl groupoid} $\mc{G}_{(A,B)}$ is the quotient of
    \[
        D:=
        \left\{
            (\alpha_{n}(x),n, x) \mid
            n\in N(B),
             x\in \mathrm{dom}(n)
        \right\}
    \]
    by the equivalence relation
    \begin{align*}
    \begin{split}
       & (\alpha_n(x), n, x) \sim
        (\alpha_m( y), m, y)
         \iff   \\
       &x=y  
        \text{ and }
        \alpha_n\vert_{U}
        = \alpha_m\vert_{U}
        \text{ for an open neighborhood $U\sset\widehat{B}$ of $x$}
        .
    \end{split}
    \end{align*}
    We will
    denote the equivalence class of $(\alpha_n( x), n, x)$  by $[\alpha_n(x), n, x]$. 
    The  
    groupoid structure on $\mc{G}_{(A,B)}$ 
    is
    defined by 
    \begin{align*}
        [\alpha_n(\alpha_m(x)), n, \alpha_m(x)]
        \cdot
        [\alpha_m(x), m, x],
        &=
        [\alpha_{nm}(x), nm, x], \text{ and }
        \\
        [\alpha_m(x), m, x]\inv &= [x, m^{*}, \alpha_{m}(x)].
    \end{align*}
   To topologize $\mc{G}_{(A,B)}$,  we define a basic open set to be of the form $\{ [\alpha_n(x), n, x]: \alpha_n(x) \in V, x \in U\}$, where $U, V \subseteq \widehat B$ are open and $n \in N(B)$ 
   (see \cite[Section 3]{renault-cartan}).
    It follows from the remark at the bottom of page 971 in \cite{Kumjian:diag} that $x\in\dom(n)$ if and only if $\Omega\inv (n^{*}n)$ does not vanish at $x$.
     \begin{rmk}\label{rmk:dom-ran-argument}
        For each $x\in \dom(n)$, we have $\alpha_n (x)(nn^{*}) = x( n^{*}n)$. Indeed, if $b\in B$, then
        $
            x(n^{*} b n ) = \alpha_{n}(x)(b) \, x(n^{*}n).
        $
        In particular, for $b= nn^{*},$  the fact that each functional $x\in \widehat B$ is multiplicative implies that
        \[
            x(n^{*}n)\,x(n^{*}n) = x(n^{*} (nn^{*}) n )  =  \alpha_{n}(x)(nn^{*}) \cdot x(n^{*}n).
        \]
        Since $x\in\dom(n)$, we can divide by $x(n^{*}n )$ and get
        $ 0\neq x(n^{*}n)= \alpha_{n}(x)(nn^{*}) $ as claimed.
    \end{rmk}
    
    The {\em Weyl twist} $ \Sigma_{(A,B)}$ is another groupoid associated to the Cartan pair $(A,B)$. Like the Weyl groupoid, the Weyl twist is also a quotient of $D$,  but by the equivalence relation
    \begin{align*}
        &(\alpha_n(x), n, x) \approx
        (\alpha_m(y), m, y)  
        \iff  \\
        &
        x = y
        \text{ and }
        \exists \ b,b'\in B \text{ such that }
        x(b),x(b') >0 \text{ and }
        nb=mb'
        .
    \end{align*}
    We write $\llbracket \alpha_n(x), n, x \rrbracket$ for the class of $(\alpha_n(x), n, x)$ in $ \Sigma_{(A,B)}$. We point out that equivalence with respect to $\approx$ implies equivalence with respect to $\sim$. 
    
    As its name suggests, the Weyl twist is a $\mathbb T$-groupoid, or twist, over $\mc{G}_{(A,B)}$.  As such, one can construct the twisted groupoid $C^{*}$-algebra $C_{r}^{*} \left(\mc{G}_{(A,B)},\Sigma_{(A,B)}\right)$ (cf.~\cite[Section 4]{renault-cartan}).
    The Weyl twist and groupoid are constructed exactly so that
    \[
        (A,B) \cong \left(C_{r}^{*} \left(\mc{G}_{(A,B)},\Sigma_{(A,B)}\right), C_{0}\bigl(\mc{G}_{(A,B)}\z\bigr)\right),
    \]
    see \cite[Theorem 5.9]{renault-cartan}.

 We recall the construction of a twisted groupoid $C^{*}$-algebra in the case when the twist arises from a 2-cocycle, 
as this is the level of generality we will need in this paper.  Recall that a {\em \textup($\mathbb{T}$-valued\textup) $2$-cocycle} on a groupoid $\gpd $ is a 
function $c\colon \gpd\2\to\mbb{T}$ 
which satisfies the {\em cocycle condition}
\begin{equation}\label{eq:cocycle}
    c(x,yz)\, c(y,z) = c(x,y)\, c(xy,z) 
    \text{ for all } (x,y), (y,z)\in\gpd\2.
\end{equation}
Given a second countable, locally compact Hausdorff, {\'e}tale groupoid $\gpd$ and a continuous  2-cocycle $c\colon  \gpd\2 \to \mathbb T$, we denote by $C_c(\gpd, c)$ the collection of continuous $\mathbb C$-valued functions on $\gpd$ which we view as a $*$-algebra via the twisted convolution multiplication
    \[ f g\, (\gamma) : = \sum_{\eta \rho = \gamma} f(\eta) g(\rho) c(\eta, \rho)\]
    and the involution 
    \[ f^{*}(\gamma) := \overline{f(\gamma\inv ) c(\gamma, \gamma\inv )}.\]
    For each $u \in \gpd\z$, we represent $C_c(\gpd, c)$ on $\ell^2(\gpd_u) $ by left multiplication:
    \[\pi_u(f)\xi :=
    f\xi=
    \left( \gamma \mapsto \sum_{\eta \rho = \gamma} f(\eta) \xi(\rho) c(\eta, \rho)\right).\] Let $\| f\|_u $ denote the operator norm of $ \pi_u(f).$
    The {\em reduced twisted groupoid $C^{*}$-algebra} $C^{*}_r(\gpd, c)$ is then the completion of $C_c(\gpd, c)$ in the norm $\| \cdot \|_r := \sup_{u\in \gpd\z} {\| \cdot \|_u}$.

While the definition of the Weyl groupoid and Weyl twist seem quite different, the following very helpful proposition describes the groupoid in terms more similar to the twist.  The result may be known to experts, but we were unable to locate it in the literature.
\begin{prop}
\label{prop:Weyl-gpd-similar-to-twist}
    Suppose $n_{i}\in N(B)$ and $x\in \dom(n_{1})\cap\dom(n_{2})$. Then $[\alpha_{n_{1}}(x), n_{1}, x]=[\alpha_{n_{2}}(x), n_{2}, x]$ if and only if there exist $b_{i}\in B$ so that $x (b_{i}) \neq 0$ and $n_{1}b_{1}=n_{2}b_{2}.$
\end{prop}
For the proof, we require two lemmata.

\begin{lemma}\label{lem:alpha-same-iff-in-B}
    If $n,m\in N(B)$, then
    \[
        \alpha_{n}\equiv \alpha_{m} \text{ on }\dom (n)\cap\dom (m)
        \iff
        nm^{*} \in B.
    \]
\end{lemma}

\begin{proof}
    We know that $nm^{*}\in B$ if and only if $\alpha_{nm^{*}}=\mathrm{id}_{\dom (nm^{*})}$ since $B$ is maximal Abelian. Moreover, $\alpha_{nm^{*}} = \alpha_{n} \circ \alpha_{m^{*}} = \alpha_{n}\circ\alpha_{m}\inv$. Both things combined yield that $nm^{*}\in B$ if and only if $\alpha_{n}\circ\alpha_{m}\inv$ is the identity on
    $
        \dom (nm^{*})
    .
    $
    Since $n^{*}n \in B$, we can use the defining property of $\alpha_{m^{*}}$ (Equation~\eqref{eq:defining-eq-for-alpha}) to rewrite 
    \[
        (mn^{*}nm^{*})(x)
        =
        (n^{*}n \circ \alpha_{m^{*}}) (x) mm^{*}(x),
    \]
    so we have
    \begin{align*}
        \dom (nm^{*})
        =&
        \left\{
            x\in \widehat{B} \,\vert\, x\in\dom (m^{*}) \text{ and } (n^{*}n \circ \alpha_{m^{*}}) (x) \neq 0
        \right\}
        \\=&
        \dom (m^{*})\cap \alpha_{m^{*}}\inv\bigl( \dom (n)\bigr)
        =
        \alpha_{m}\bigl( \dom (m)\cap  \dom (n)\bigr)
        .
    \end{align*}
    Thus, $ nm^{*}\in B$ if and only if $\alpha_{n}\circ\alpha_{m}\inv$ is the identity on $\alpha_{m}\bigl( \dom (m)\cap  \dom (n)\bigr)$ if and only if $\alpha_{n} \equiv \alpha_{m}$ on $ \dom (m)\cap  \dom (n)$.
\end{proof}

We define the \emph{open support} of $k \in C_0(\widehat{B})$ by $\supp'(k) := \{x \in \widehat{B}\,\vert\, k(x) \neq 0\}$.

\begin{lemma}[Urysohn-type Lemma]
\label{lem:cut-down}
    Let $f\in N(B)$ and suppose that $k \in C_0 (\widehat{B})$ has $\supp'(k)\sset \mathrm{dom}(f^{*})$. Then the partial homeomorphism associated to $\Omega(k)f$ has domain $\alpha_{f}\inv(\supp'(k)) = \alpha_{f^{*}}(\supp'(k))$, where it coincides with $\alpha_{f}$.
\end{lemma}

\begin{proof}
   First note that $f_{2}:= \Omega(k) f$ is still a normalizer of $B$ because $f$ is and because $\Omega(k)\in B$; thus, it makes sense to speak of the corresponding partial homeomorphism $\alpha_{f_{2}}$ and its domain, $\dom (f_{2})$.
   
   By definition of $f_{2},$ we have $f_{2}^{*} f_{2} = f^{*}  \Omega(\abs{k}^{2})  f.$ The defining property of $\alpha_{f}$  (Equation~\eqref{eq:defining-eq-for-alpha}, with $b=\Omega(|k|^2)$) 
   yields
        \begin{align}\label{eq2}
            \Omega\inv (f_{2}^{*} f_{2})
            =
            (\abs{k}^{2}\circ \alpha_{f}\bigr) \cdot \Omega\inv(f^{*} f).
        \end{align}
    By assumption, $V := \alpha_{f}\inv(\supp'(k)) $ is contained in $\dom(f)$, so the above yields
        \begin{align*}
            \dom(f_{2})
            :=&
            \supp' \bigl(\Omega\inv (f_{2}^{*} f_{2})\bigr)
            =
             \supp' (k \circ \alpha_{f})  \cap \supp'\bigl( \Omega\inv (f^{*} f)\bigr) 
             \notag
             \\
            =&
             V \cap \dom (f)
             =
             V.\notag 
        \end{align*}
        By the defining property of $\alpha_{f_{2}}$,  for any $b\in B$ we have
        \begin{align}\label{eq1}
            \Omega\inv (f_{2}^{*}  b f_{2})
            &= (\Omega\inv(b)\circ \alpha_{f_{2}})\cdot \,\Omega\inv (f_{2}^{*} f_{2})\notag
            \\
            &=(\Omega\inv(b)\circ \alpha_{f_{2}})\cdot \,  (\abs{k}^{2}\circ \alpha_{f}\bigr)\ \cdot \Omega\inv(f^{*} f),
        \end{align}
        where the second equality is due to Equation~\eqref{eq2}. The definition of $f_{2}$ implies that
        \begin{align*}
            f_{2}^{*}  b f_{2} =  f^{*} (\Omega(k)^{*} b  \Omega(k) ) f ,
        \end{align*}
        so that the defining property of $\alpha_{f}$ (Equation~\eqref{eq:defining-eq-for-alpha})  yields
        \begin{align*}
            \Omega\inv (f_{2}^{*}  b f_{2})
            =&
            \bigl( (\overline{k} \cdot \Omega\inv(b) \cdot k )\circ \alpha_{f}\bigr) \cdot \Omega\inv(f^{*} f)
            \\
            =&
            (\Omega\inv(b)\circ \alpha_{f}) \cdot (\abs{k}^2\circ \alpha_{f} ) 
            \cdot \Omega\inv(f^{*} f).
        \end{align*}
        Combining this with Equation~\eqref{eq1} reveals that, for any $b \in B,$
        \[
            (\Omega\inv(b)\circ \alpha_{f_{2}})\cdot \,(\abs{k}^{2}\circ \alpha_{f}\bigr) \cdot \Omega\inv(f^{*} f)
            =
            (\Omega\inv(b)\circ \alpha_{f}) \cdot (\abs{k}^2\circ \alpha_{f} )
            \cdot \Omega\inv(f^{*} f).
        \]
        We conclude that $\alpha_{f_{2}}  = \alpha_{f}$ wherever neither $\Omega\inv(f^{*} f)$ nor $(\abs{k}^2\circ \alpha_{f} )$ vanish, which occurs precisely in $\dom(f) \cap \supp' (k\circ \alpha_{f} ) = \dom(f_{2})$.
\end{proof}

\begin{proof}[Proof of Proposition~\ref{prop:Weyl-gpd-similar-to-twist}]
    First, fix $x \in \dom(n_1) \cap \dom(n_2)$ and assume that $n_{1}b_{1}=n_{2}b_{2}$ for some $b_{i}\in B$ such that $x(b_{i})\neq 0$. In particular, there exists a neighborhood $X$ of $x$ in $\dom(n_{1})\cap\dom(n_{2})$ so that, for all $y\in X$, we have $y(b_{i})\neq 0$. 
    If $g$ is any element of $C_{0}(\widehat{B})$, then by the defining property of $\alpha_{n_{i}}$ (Equation~\eqref{eq:defining-eq-for-alpha})  and since $B$ is commutative, we have
    \begin{align*}
        (n_{i}b_{i})^{*} \,\Omega(g) \,(n_{i}b_{i})
        &=
        b_{i}^{*} \,(n_{i}^{*} \Omega(g) n_{i})\, b_{i}
        \\&=
        b_{i}^{*} \Omega(g\circ \alpha_{n_{i}}) (n_{i}^{*} n_{i}) b_{i}
        =
         \Omega(g\circ \alpha_{n_{i}}) (n_{i}^{*} n_{i})(b_{i}^{*}b_{i}),
    \end{align*}
    so that the equality $n_{1}b_{1}=n_{2}b_{2}$ implies that 
    \[
        \Omega(g\circ \alpha_{n_{1}}) (n_{1}^{*} n_{1})(b_{1}^{*}b_{1})
        =
        \Omega(g\circ \alpha_{n_{2}}) (n_{2}^{*} n_{2})(b_{2}^{*}b_{2}).
    \]
    Evaluating both sides at $y\in X$ yields
    \[
       g\bigl( \alpha_{n_{1}}(y)\bigr) \, y(n_{1}^{*} n_{1})\, \abs{y(b_{1})}^2
        =
        g\bigl( \alpha_{n_{2}}(y)\bigr) \, y(n_{2}^{*} n_{2})\, \abs{y(b_{2})}^2.
    \]
    By our construction of $X$, we have $y(b_{i})\neq 0$ and also $X\sset\dom(n_{i})$, so that $y(n_{i}^{*} n_{i}) >0$. We have shown that, for any $g\in C_0 (\widehat{B})$, $g\left( \alpha_{n_{1}}(y)\right)$ is a %
    positive multiple of $g\left( \alpha_{n_{2}}(y)\right) $. Since $C_0(\widehat{B})$ separates points (as $\widehat{B}$ is Hausdorff), this implies  $\alpha_{n_{1}}(y)=\alpha_{n_{2}}(y)$ for all $y\in X$. As $X$ is an open neighborhood of $x$, we arrive at the claimed equality in the Weyl groupoid.

    Conversely, assume that
    \[ 
        [\alpha_{n_{1}}(x), n_{1}, x]=[\alpha_{n_{2}}(x), n_{2}, x].
    \]
    We will construct $b_1, b_2 \in B$ such that $x(b_i)\neq 0$
    and  $n_{1}b_{1}=n_{2}b_{2}$.

    By assumption, there exists a neighborhood $X$ of $x$ in $\dom(n_{1})\cap \dom(n_{2})$ on which $\alpha_{n_1}$ and $\alpha_{n_2}$ agree. Let $Y := \alpha_{n_1}(X) = \alpha_{n_2}(X)$, and note that $\alpha_{n_1^{*}}(Y) = \alpha_{n_2^{*}}(Y) = X$.
    As $X$ is an open neighborhood of $x$, \cite[Proposition 4.32 (Urysohn's Lemma)]{Folland} implies the existence of $k\in C_0 (\widehat{B})$ with $k (x)=1$ and $\supp'(k)\sset X$. By our choice of $X$,
    \begin{equation}\label{eq:consequence-of-supp'k}
        y\in\supp'(k)
        \implies 
        \alpha_{n_{1}}(y)
        =
        \alpha_{n_{2}}(y),
        \text{ i.e.\ }
        y=
        \alpha_{n_{1}^{*}}\bigl( \alpha_{n_{2}}(y)\bigr).
    \end{equation}
    By Lemma~\ref{lem:cut-down}, we know that
    $$m_{i}:= \Omega(k)n_{i}^{*}\in N(B)$$
    has $\dom(m_i)= \alpha_{n_{i}} (\supp'(k)) \sset Y \sset \dom(n_{i}^{*})$, and that $\alpha_{m_{i}}$ is extended by $\alpha_{n_{i}^{*}}.$ In particular,  it follows from Implication~\eqref{eq:consequence-of-supp'k} that
    \[
        \dom(m_{1})=\alpha_{n_{1}} (\supp'(k))\overset{\eqref{eq:consequence-of-supp'k}}{=}\alpha_{n_{2}} (\supp'(k))=\dom(m_{2}) \sset Y.
    \]
    This means that $\alpha_{m_{1}}=\alpha_{n_{1}^{*}}\vert_{\dom(m_{1})}=\alpha_{n_{2}^{*}}\vert_{\dom(m_{2})}=\alpha_{m_{2}}$ on all of $\dom(m_{1})=\dom(m_{2})$. By Lemma~\ref{lem:alpha-same-iff-in-B}, we conclude that
    \begin{equation}\label{eq:def:b_1}
        b_{1}:=m_{1}m_{2}^{*} = \Omega(k) n_1^{*} n_2\Omega(\overline{k})
    \end{equation}
    is an element of $B$. To see that $x(b_1) \neq 0$, use the defining property of $\alpha_{m_2^{*}}$ (Equation~\eqref{eq:defining-eq-for-alpha})  to write 
    \begin{align*}
        x(b_{1}^{*}b_{1})
        &=
        x(m_{2}m_{1}^{*}m_{1}m_{2}^{*})
        =
        x(m_{2}m_{2}^{*}) \cdot \alpha_{m_{2}^{*}}(x)(m_{1}^{*}m_{1}).
    \end{align*}
    Since $\supp'(k) \subseteq \dom(n_i)$ and $\alpha_{n_i^{*}}$ extends $\alpha_{m_i}$, 
    \[\supp'(k) = \alpha_{n_i^{*}}(\alpha_{n_i}(\supp'(k))) = \alpha_{n_i^{*}}(\dom(m_i)) = \alpha_{m_i}(\dom(m_i)) = \dom(m_i^{*}).
    \]
    Therefore, $x \in \supp'(k) = \dom(m_2^{*})$ implies $x(m_2m_2^{*})>0$. Moreover, $\alpha_{m_2^{*}}(x) = \alpha_{m_1^{*}}(x) \in \dom(m_1)$, so $\alpha_{m_2^{*}}(x)(m_1^{*}m_1) >0$. Consequently, $x(b_1^{*}b_1)>0$, so $x(b_1) \neq 0$.
    
    By the defining property of $\alpha_{n_{1}^{*}}$, we have
    \begin{equation}\label{eq:n1b1}
        n_{1}b_{1} 
        =
        n_{1} (\Omega(k)n_{1}^{*}m_{2}^{*})
        =
        \Omega(k\circ\alpha_{n_{1}^{*}}) \,(n_{1}n_{1}^{*})\, m_{2}^{*} = \Omega(k \circ \alpha_{n_1^{*}} ) (n_1 n_1^{*}) n_2 \Omega(\overline{k}).
    \end{equation}
Our goal is to rewrite the right-hand side of Equation~\eqref{eq:n1b1} in the form $n_2 b_2$ for some $b_2 \in B$ such that $x(b_2)\neq0$. Note that if  
    \[f:=(k\circ\alpha_{n_{1}^{*}})\cdot\Omega\inv(n_{1}n_{1}^{*}) ,\]
    then $f \in C_0 (\widehat{B})$,  Equation~\eqref{eq:n1b1} becomes $n_1 b_1  = \Omega(f) m_2^{*}$, and $f$ is supported in $\alpha_{n_1}(\supp'(k)) = \alpha_{n_{2}}(\supp'(k)) \sset Y$.

As $\supp'(k) \subseteq X \subseteq \dom(n_2)$, and $\alpha_{n_2} \circ \alpha_{n_2^{*}}(y) = y$ for all $y \in Y = \alpha_{n_2}(X)$, we have $n_2 \Omega(f \circ \alpha_{n_2} ) = \Omega(f) n_2$ by \cite[Lemma 4.2]{DGNRW:Cartan}.
Equation~\eqref{eq:n1b1} can therefore be rewritten as  
\[ n_1 b_1 = n_2 \Omega(f \circ \alpha_{n_2}) \Omega(\overline{k}).\]
Setting $b_2 = \Omega( (f \circ \alpha_{n_2}) \overline{k})$, we have $n_1 b_1 = n_2 b_2$.  

We now complete the proof by showing that $x(b_2) > 0$.  As $k(x) = 1$ by construction, it suffices to show that $f(\alpha_{n_2}(x)) > 0$.  Our construction of $X \ni x$ implies that $k( \alpha_{n_1^{*}} (\alpha_{n_2}(x))) =k( x) =1$ and that 
\[ \Omega\inv(n_1n_1^{*})(\alpha_{n_2}(x)) = \Omega\inv(n_1 n_1^{*})(\alpha_{n_1}(x))= x(n_1^{*} n_1) > 0,\] where the last equality follows from Remark~\ref{rmk:dom-ran-argument}.
Thus, 
$f(\alpha_{n_2}(x))=[(k\circ\alpha_{n_{1}^{*}})\cdot\Omega\inv(n_{1}n_{1}^{*})](\alpha_{n_2}(x))>0$, as desired. 
\end{proof}

\pagebreak[3]\section{The spectrum of a twisted bundle of groups}\label{sec:abelian-gp-bdl}

Assume that $\agpd $ is a second countable, locally compact Hausdorff, {\'e}tale groupoid
and that $c\colon \agpd^{(2)} \to \mathbb{T}$ is a $2$-cocycle on $\agpd$.
We will always assume that $2$-cocycles are {\em normalized}, i.e.\ $c(r(a),a)=1=c(a,s(a))$ for each $a\in\agpd $. 
In order to construct the twisted groupoid $C^{*}$-algebra $C^{*}_r(\agpd, c),$
we will need $c$ to be continuous, so we will frequently impose this assumption.

In this section, we will be interested in bundles of groups, so on top of our topological assumptions above, assume that the range and source maps of $\agpd $ are equal, called ${ p }\colon\agpd \to\agpd\z$.
We write $\agpd_{u}:={ p }\inv(\{u\})$ for $u$ a unit. Moreover, assume that the multiplication map $\agpd\2 \to \agpd$ is commutative
and that the continuous $2$-cocycle $c$ is symmetric on $\agpd $, i.e.\ $c(a,a')=c(a',a)$ for all $a,a'\in\agpd_{u}$, so that its  reduced twisted $C^{*}$-algebra $B:= C^{*}_r (\agpd ,c)$ is commutative by \cite[Lemma 3.5]{DGNRW:Cartan}.

\begin{rmk}\label{rmk:mcA-amenable}
 Observe 
    that $\agpd $ is amenable. Indeed, since $\agpd $ is a bundle of Abelian groups, $C_{r}^{*}(\agpd )$ is commutative; in particular, it is nuclear. Since $\agpd $ is {\'e}tale, amenability therefore follows from \cite[Thm.\ 4.1.5]{AidansNotes}. In  particular, \cite[Corollary 4.3]{jon-astrid} implies that  $C^{*}(\agpd ,c) \cong C^{*}_{r}(\agpd ,c)$.
\end{rmk}

\begin{defn}
Given  $u \in \agpd\z$ and a continuous 2-cocycle $c$ on $\agpd$, let $\specB_{u}$ denote the set of 1-dimensional $c$-projective representations of the Abelian group $\agpd_{u}$. That is, $\specB_{u}$ consists of the maps $\chi\colon \agpd_{u} \to  \mbb{T}$ such that 
\begin{equation}
    \label{eq:chi-proj}
\chi(a) \chi(a') = c(a, a') \chi(aa').\end{equation}
(Observe that such maps $\chi$ necessarily satisfy $\chi(u) = 1$, because $c(a, u) = 1= c(u, a)$ for all $a\in\agpd_{u}$.)
Write $\specB  = \bigsqcup_{u \in  \agpd \z} \specB_{u}$ and $\rho\colon \specB  \to  \agpd \z$ for the projection map. 
\end{defn}

Recall that the {\em spectrum} $\widehat C$ of a commutative $C^{*}$-algebra $C$ is the set of nondegenerate one-dimensional representations of $C$, equipped with the weak-$*$ topology.
As $B= C^{*}_r (\agpd ,c)$ is commutative, the Gelfand--Naimark Theorem yields $B \cong C_0(\widehat B)$. We will show that $\widehat{B} \cong \specB $ for a suitable topology on $\specB$.

\begin{lemma}\label{lem:phi-chi-alg-repn} 
    For $\chi\in \specB $, let
    \[
        \phi_{\chi}\colon C_c(\agpd , c) \to  \mbb{C},
        \quad
        \phi_\chi(f) := \sum_{a \in \agpd_{\rho(\chi)}} \chi(a) f(a).
    \]
    Then $\phi_{\chi}$ is a $*$-algebra homomorphism and extends to an element of $\widehat{B}$.
\end{lemma}
 At times, we will write $\phi(\chi)$ instead of $\phi_\chi$, to ease notation.

\begin{rmk}
   It is unclear to the authors whether the formula for $\phi_{\chi}$ in Lemma~\ref{lem:phi-chi-alg-repn} extends to elements of $B$ when thought of as $C_0$-functions on $\agpd $. In particular, if $b\in B\sset C_{0} (\agpd )$ with $\phi_{\chi}(b)\neq 0$, does it then follow that $\supp'(b)\cap \agpd_{\rho(\chi)}\neq \emptyset$?
\end{rmk}

\begin{proof}
    First note that since $\agpd $ is {\'e}tale and second countable,  $\agpd_{u}$ is countable, so $\phi_\chi(f)$ is a finite sum for any $\chi \in \specB $ and any $f \in C_c(\agpd , c)$.

    It is evident that $\phi_\chi$ is linear. To see that $\phi_\chi$ is multiplicative, note that if $u:=\rho(\chi)$, then
    \begin{align*}
        \phi_\chi (f g) &
        = \sum_{a_{1},a_{2}  \in \agpd_{ u } } \chi(a_{1} ) f(a_{1} a_{2}) g(a_{2}\inv ) c(a_{1} a_{2}, a_{2}\inv ).
    \end{align*}
    Using the defining formula~\eqref{eq:chi-proj} for $\chi$,
     \begin{align*}
        \phi_\chi(f g) 
        &= \sum_{ a_{1},a_{2} \in \agpd_{ u } } \chi(a_{1} a_{2}) \chi(a_{2}\inv ) f(a_{1} a_{2}) g(a_{2}\inv ) 
        = \phi_\chi(f) \phi_\chi(g).
    \end{align*}
    
 To see that $\phi_\chi$ is $*$-preserving, we use Equation~\eqref{eq:chi-proj} to compute
    \begin{align}\label{eq:chi-a-inv}
    \chi (a) \overline{c(a,a\inv)} = \overline{\chi (a\inv)}\chi \bigl({ p }(a)\bigr) = \overline{\chi (a\inv)}.
 \end{align}
    Thus
    \[ \phi_\chi(f^{*}) = \sum_{a\in \agpd_{ u } } \chi(a) \overline{f(a\inv ) c(a, a\inv )} = \sum_{a\in\agpd_{ u } } \overline{\chi(a\inv ) f(a\inv )} = \sum_{a\in \agpd_{ u } } \overline{\chi(a) f(a)}.\]

    To see that $\phi_{\chi}\colon C_c(\agpd , c) \to \mbb{C}$ extends to an element of $\widehat{B}$, i.e.\ a non-zero norm-decreasing $*$-homomorphism $\phi_\chi\colon C_{r}^{*}(\agpd , c) \to \mbb{C}$, recall that the full $C^{*}$-algebra $C^{*}(\agpd , c)$ is the completion of $C_c(\agpd , c)$ in the norm \begin{equation}
    \|f\| := \sup \{ \| L(f)\|: L \text{ a $*$-representation of $C_c(\agpd , c)$, } \| L(f)\| \leq \|f\|_{I}\}.
    \label{eq:full-norm}
    \end{equation}
   In our case, where $\agpd  = \Iso(\agpd ),$ we have
   $\|f\|_{I} = \sup_{u \in  \agpd \z} \sum_{\agpd_{u}} |f(x)|$ for $f \in C_c(\agpd , c).$
   We have already shown that
   each $\phi_\chi$ is a *-representation of $C_c(\agpd , c)$, and by the triangle inequality, it is $I$-norm decreasing, i.e.\ one of the representations invoked in Equation \eqref{eq:full-norm}. In particular, $\| \phi_\chi(f)\| \leq \| f\|$ for any $f \in C_c(\agpd , c)$ and any $\chi \in \specB ,$ so the representation $\phi_\chi$
    extends to a nonzero  $*$-homomorphism $\phi_\chi\colon C^{*}(\agpd , c) \to \CC$. The claim now follows because $C^{*}(\agpd ,c) \cong C^{*}_{r}(\agpd ,c)$ by Remark~\ref{rmk:mcA-amenable}.
\end{proof}

\begin{lemma}
\label{lem:phi-surjective}\label{lem:phionetoone}
    The map $
    \phi\colon \specB \to \widehat{B}, \, \chi \mapsto \phi_\chi,$ is a bijection.
\end{lemma}

\begin{proof}
A moment's thought reveals that the map $\chi \mapsto \phi_\chi$ is injective: If $\chi \not= \chi'$ while $\rho(\chi)=\rho(\chi')$, there must exist $a \in \agpd_{\rho(\chi)}$ with $\chi(a) \not= \chi'(a)$.
Then, if  $f \in C_c(\agpd, c )$ is supported in a bisection containing $a$ and $f(a) = 1$, we have 
\[ \phi_\chi(f) = \chi(a) f(a) = \chi(a) \not= \chi'(a) = \phi_{\chi'}(f),\]
so $\phi_\chi \not= \phi_{\chi'}$.
Therefore, to show that $\phi$ is a bijection, it suffices to show that all one-dimensional representations of $C^{*}_r(\agpd_{u}, c)$ are of the form $\phi_\chi$ for some $\chi\in\specB$.

To that end, we observe 
that $B$ is a $C_0(\agpd\z)$-algebra (cf.~\cite[Remark 5.1]{crytser-nagy}).  Then \cite[Proposition C.5]{xprod} implies that
$\widehat B = \bigsqcup_{u \in \agpd\z} \widehat{B(u)}$.  When $c$ is trivial, \cite[Remark 5.2]{crytser-nagy} and Remark \ref{rmk:mcA-amenable} imply that $B(u) \cong C^{*}_r(\agpd_{u}, c)$. 
In fact,
a careful examination of the proof of \cite[Remark 5.2]{crytser-nagy} reveals that, even if $c$ is not trivial, the formulae used there will also give an isomorphism between the fiber algebra $B(u) $ and the twisted group $C^{*}$-algebra $C^{*}_r(\agpd_u, c)$.
 
It is a classical fact that representations of the twisted group $C^{*}$-algebra $C^{*}(\agpd_{u}, c) \cong C^{*}_r(\agpd_{u}, c) \cong B(u)$ are in bijection with unitary projective representations  of $\agpd_{u}$. 
In particular,  the one-dimensional representations of $B(u)$ are of the form
\[ \pi(f) = \sum_{a \in \agpd_{u}}  \chi(a) f(a) \quad \text{for all } f\in C_c (\agpd_{u},c)\]
for some $\chi \in \specB_{u}$, see the formula in \cite[Theorem 3.3(2)]{BS:Reps-twisted-gp}. 
Comparing the above with the definition of $\phi$, we see that indeed, the map $\chi \mapsto \phi_\chi$ is surjective, and hence a bijection.
\end{proof}

\begin{prop}\label{shortened-prop:mcB-cong-widehatB}
    If we equip $\specB$ with the topology induced by $\phi$ from $\widehat{B}$, then a net $(\chi_i)_i$ converges to $\chi$ in $\specB$ if and only if the following two conditions hold:
    \begin{enumerate}[label=\textup{(\arabic*)}] 
    \item\label{item1:convergence-in-specB}  $\rho(\chi_{i}) \to \rho(\chi)$ in $\agpd\z$; 
    \item\label{item2:convergence-in-specB} whenever $( a_{i} )_{i} \subseteq \agpd $ satisfies ${ p }( a_{i} ) = \rho(\chi_{i}),$ ${ p }(a) = \rho(\chi)$, and $ a_{i} \to a $ in $\agpd $, then $\chi_{i}( a_{i} ) \to \chi(a)$ in $\mbb{T}$.
    \end{enumerate} 
\end{prop}

\begin{rmk}
    Note that, with respect to this topology on $\specB$, $\rho\colon \specB\to \agpd\z$ is clearly continuous. Note further that, when $c$ is trivial, this result is well-known; the description of the topology on $\specB$ should be compared to \cite[Proposition 3.3]{MRW96}.
\end{rmk}

  The proof of Proposition \ref{shortened-prop:mcB-cong-widehatB} proceeds through a series of lemmata.

\begin{lemma}
\label{lem:not-in-support}
    Suppose $(\chi_i)_{i\in I}$ and $\chi$ are elements of $\specB$ %
    satisfying Condition~\ref{item1:convergence-in-specB}
    of 
    Proposition~\ref{shortened-prop:mcB-cong-widehatB}. If $f\in C_c (\agpd,c)$ is supported in a bisection and $\supp'(f)\cap \agpd_{\rho(\chi)}=\emptyset$, then $\phi_{\chi_i}(f)\to 0.$
\end{lemma}

\begin{proof}
    Assume for a contradiction that 
    $\phi_{\chi_i}(f)\not\to 0$, i.e.\ 
    there exists $\epsilon>0$ so that for all $i\in I $, there exists $g(i)\in  I $ with $g(i)\geq i$ such that
    \(  \abs{\phi_{\chi_{g(i)}}(f)}
        >\epsilon.
    \) In particular, for each $ j $ in  
    \(
        J:= \{ g(i) \,\vert\, i\in I \},
    \)
    we have with $u_i :=\rho(\chi_i)$, 
    \begin{align*}
\epsilon < \abs{\phi_{\chi_{ j }}(f)}  &= \abs{\sum_{ a \in \agpd_{u_{ j }}} \chi_{ j }(a) f(a)} =\begin{cases}
|f(a_{ j })|, & \{a_{ j }\} = \supp'(f) \cap \agpd_{u_{ j }} \neq \emptyset \\
0, & \text{otherwise},
\end{cases}
\end{align*}
so there must exist an element $a_j $ in $\supp'(f)\cap \agpd_{ u_j }$.
    
    Note that $J$ is a directed set when equipped with $ I $'s preorder: if we take $ j_{1}, j_{2}\in J,$ then since $ I $ is directed, there exists $i \in I$ with $i \geq  j_{1}, j_{2}$. Then $g(i)\geq i$, so $g(i)$ is an upper bound for $ j_{1}$ and $ j_{2}$ in $J$. Firstly, this implies that $(u_{ j })_{ j \in J}$ is a subnet of $(u_{i})_{i\in I }$ (the inclusion $J\hookrightarrow  I $ is monotone and final),
    so that
    Condition~\ref{item1:convergence-in-specB} of Proposition~\ref{shortened-prop:mcB-cong-widehatB} implies $\lim_{ j \in J}u_{ j }= \lim_{i\in I } u_i = \rho(\chi)$. 
    Secondly,
    we conclude
    that $(a_{ j })_{ j \in J}$ is a net in $\supp'(f)$. Since $\supp(f)$ is compact, there exists a subnet $(a_{\kappa})_{\kappa\in K}$ of $(a_{ j })_{ j \in J}$ which converges to some $a\in \supp(f)$. By continuity of $f$, we have
    \[
        \abs{f(a)} = \lim_{\kappa} \abs{f(a_{\kappa})} \geq \epsilon,
    \]
    i.e.\ $a\in \supp'(f)$. But 
    \(p(a)=\lim_{\kappa} p(a_\kappa) = \lim_{\kappa} u_{\kappa} = \rho(\chi),\) i.e.\ \(
a\in \agpd_{ \rho(\chi)}
    \)
    also, which contradicts our hypothesis that $\supp'(f) \cap \agpd_{\rho(\chi)} = \emptyset.$
\end{proof}

\begin{lemma}\label{lem:phi-continuous}
    Suppose $(\chi_i)_{i\in I}$ and $\chi$ are elements of $\specB$    satisfying Conditions~\ref{item1:convergence-in-specB} and~\ref{item2:convergence-in-specB} of
    Proposition~\ref{shortened-prop:mcB-cong-widehatB}. Then $\phi_{\chi_i}\to\phi_\chi$ in $\widehat{B}$. 
\end{lemma}

\begin{proof}
We must show that, 
for all $f \in C_c(\agpd ,c)$ and for all $ \epsilon >0$, there exists  $i_0 
\in I
$  such that, if $i \geq i_0$, then $|\phi_{\chi_i}(f) - \phi_\chi(f)| < \epsilon$.

We will begin by proving the claim for $f \in C_c(\agpd ,c)$ such that $\supp(f)$ is a bisection. 
Let $u:=\rho(\chi)$ and $u_i :=\rho(\chi_i)$. 
If $\supp'(f) \cap \agpd_{u} = \emptyset$, then $\phi_\chi (f)=0$ and Lemma~\ref{lem:not-in-support} yields $\phi_{\chi_i}(f) \to 0 = \phi_\chi(f)$, as claimed. Otherwise,    
fix $\epsilon >0$ and let $a \in \supp'(f)$ such that $p(a)=u$. Since $f(a) \neq 0$ and $f$ is continuous,  there exists an open neighborhood $V$ of $a$ on which $f$ is nonzero.
In fact, $V$ is a bisection around $a$ because $V \subseteq \supp'(f)$, and $p(V)$ is an open neighborhood of $u$ because $\agpd$ is {\'e}tale. As $u_i\to u$, by shrinking the neighborhoods $V$ we see that $\{a_i\in V\,\vert\, p(a_i)=u_i\}$ is a net in $\agpd$ converging to $a$.
In this case,
\begin{align*}
\abs{\phi_{\chi_i}(f) - \phi_\chi(f)}
&= \abs{ \sum_{a_i' \in \agpd_{u_i}} \chi_{i}(a_i') f(a_i') - \sum_{a' \in \agpd_{u}} \chi(a') f(a')}
\\
&= \abs{\chi_i(a_i)f(a_i) - \chi(a)f(a)}.
\end{align*}
If we now use the fact that $f \in C_c(\agpd, c)$ is bounded in $\norm{\cdot}_{\infty}$
and our hypothesis that $(\chi_i)_i$ and $\chi$ satisfy Condition~\ref{item2:convergence-in-specB} of Proposition~\ref{shortened-prop:mcB-cong-widehatB}, an easy $\epsilon/2$-argument establishes that $|\phi_{\chi_i}(f) - \phi_\chi(f)| < \epsilon $ for $i \geq i_0$ for some $i_0 \in I$.

For more general functions, recall that since $\agpd $ is a second countable, locally compact Hausdorff, {\'e}tale groupoid, we have
\[ C_c(\agpd ,c) = \text{span}\{f \in C_c(\agpd ,c) \mid \supp(f) \text{ is a bisection}\},\]
see~\cite[Lemma 3.1.3]{AidansNotes}.
An $\epsilon/k$-argument now shows that, for any $g \in C_c(\agpd, c),$ there exists $i_1\in I$ so that $i\geq i_1$ implies $|\phi_{\chi_i}(g) - \phi_\chi(g)| < \epsilon$. \qedhere
\end{proof}

\begin{lemma}\label{lem:phi-open}
    Let $(\chi_{i})_{i\in\Lambda}$ be some net in $\specB $ such that $\phi_{\chi_{i}} \to \phi_\chi$ for some $\chi\in\specB $. Then 
    $(\chi_i)_i$ and $ \chi$ satisfy Conditions~\ref{item1:convergence-in-specB} and~\ref{item2:convergence-in-specB} of Proposition~\ref{shortened-prop:mcB-cong-widehatB}.
\end{lemma}

\begin{proof}
    Recall that our assumption $\phi_{\chi_{i}} \to \phi_\chi$ means that, for every $f \in C_c(\agpd , c)$ and $\epsilon >0,$ there exists $N_{f, \epsilon} \in \Lambda$ such that  if $i \geq N_{f, \epsilon}$ then $| \phi_{\chi_{i}}(f) - \phi_\chi(f)| < \epsilon$.

We start by proving \ref{item1:convergence-in-specB} of Proposition~\ref{shortened-prop:mcB-cong-widehatB}. Recall first that the unit space in an {\'e}tale groupoid is clopen (\cite[Lemma 2.4.2]{AidansNotes}), so that $u_{i}\to u$ in $\agpd\z$ if and only if $u_{i}\to u$ in $\agpd $. 
Let $V\subseteq \agpd \z$  be an open neighborhood of $u.$
Since $\agpd $ is locally compact Hausdorff and $V$ is an open neighborhood of $u$ in $\agpd $, there exists by 
  \cite[4.32 (Urysohn's Lemma)]{Folland} a function $f$ in $C_{c}(\agpd ,c)$ with $f (u) = 1$ and $f_{\vert \agpd \setminus V}\equiv 0$. Since $\phi_{\chi_{i}} \to \phi_\chi$, then for any fixed $1 > \epsilon >0,$ there exists an $M\in\Lambda $ such that if $i \geq M$, then
\[
| \phi_{\chi_{i}}(f) - \phi_\chi(f)| < \epsilon < 1,\]
which, by definition of $\phi$, implies
\[
\abs{ \sum_{a' \in \agpd_{u_{i}}} \chi_{i}(a')f(a') - \sum_{a' \in \agpd_{u}} \chi(a')f(a') } <1.\]
As 
$\agpd\z \supseteq V$
is a bisection
containing $\supp'(f)$, and $\xi(v) = 1$ for any $\xi \in \specB$ and $v \in \agpd\z$, the above inequality becomes \[\abs{\chi_i(u_i) f(u_i) - \chi(u) f(u)} = \abs{f(u_i) - 1} < 1,\]
for all $i \geq M$.
Therefore, if $ i \geq M$, then $u_i \in \supp'(f) \subseteq V$.
This concludes the proof of \ref{item1:convergence-in-specB}.

We proceed with proving \ref{item2:convergence-in-specB}. Suppose $( a_{i} )_{i} \sset \agpd $ such that ${ p }( a_{i} ) = u_{i}$, ${ p }(a) = u$, and $ a_{i} \to a$. Fix $\epsilon >0$. We must show there exists $M \in  \Lambda $ such that if $i \geq M$, then $|\chi_{i}( a_{i} ) - \chi(a)|<\epsilon.$ By \cite[Lemma 2.4.9]{AidansNotes}, there exists an 
open bisection $W$ in $\agpd $ that contains $a$. Since $\agpd $ is locally compact Hausdorff, there exists by \cite[Proposition 4.31]{Folland} a precompact open set $U$ with $a\in U \sset \overline{U} \sset W$. Since $ a_{i} \to a$, there exists $N \in  \Lambda$ such that if $i \geq N$, then $ a_{i} \in U$.

Again by \cite[4.32 (Urysohn's Lemma)]{Folland}, there exists $f\in C_{c}(\agpd ,c)$ which is equal to $1$ on $\overline{U}$ and $0$
outside of $W$.
So for all $i$ in $\Lambda$ which are larger than both $N$ and $N_{f,\epsilon}$, 
we know $ a_{i} \in U \subseteq  \supp (f)$ and 
\[
\abs{ \sum_{a' \in \agpd_{u_{i}}} \chi_{i}(a')f(a') - \sum_{a' \in \agpd_{u}} \chi(a')f(a') }< \epsilon.
\] 
Note that $W$ is a bisection, and $ a_{i} ,a$ are elements of $U \sset W$ with ${ p }( a_{i} ) = u_{i}$ and ${ p }(a)=u$. All of these facts combined yield that $ a_{i} $ is the 
unique
element in $\agpd_{u_{i}} \cap U$ and $a$ is the unique element in $\agpd_{u} \cap U$.
Since $f$ is
equal to $1$ on $U$, the inequality becomes 
\[
|  \chi_{i}( a_{i} ) -  \chi(a) |< \epsilon.
\]
This completes the proof of the lemma, and of Proposition~\ref{shortened-prop:mcB-cong-widehatB}.
\end{proof}

\begin{cor}\label{cor:specB-is-spectrum}
    The map $\specB \to \widehat{B}, \, \chi \mapsto \phi_\chi,$ is a homeomorphism when $\specB$ is equipped with the topology described in Proposition~\ref{shortened-prop:mcB-cong-widehatB}.  In particular, $\specB$ is locally compact Hausdorff and $B$ is isomorphic to the $C^{*}$-algebra $C_0(\specB )$. 
\end{cor}

\pagebreak[3]\section{Computing the Weyl groupoid}
\label{sec:weyl-paper2}

Our standing assumptions for the remainder of this paper are the following:
\begin{enumerate}[label=\textup{(\alph*)}]
    \item\label{assumption:Sigma} $\gpd $ is a second countable, locally compact Hausdorff, {\'e}tale groupoid;
    \item\label{assumption:cocycle} $c$ is a
    normalized,
    continuous
    $\mbb{T}$-valued 2-cocycle on $\gpd $;
    \item\label{assumption:mcA} $\agpd \sset \Iso(\gpd )$ is an Abelian subgroupoid, containing $\gpd\z$, on which $c$ is symmetric;
    \item\label{assumption:mcA-clopen-normal} $\agpd $ is clopen and normal in $\gpd $; and
    \item\label{assumption:mcA-maximal} $\agpd $ is chosen such that $B:=C^{*}_r(\agpd , c)$ is maximal Abelian in $A:=	C^{*}_r(\gpd , c)$.
\end{enumerate}
Note that Assumptions~\ref{assumption:Sigma}, \ref{assumption:cocycle}, and~\ref{assumption:mcA} make sure that $(\agpd ,c)$ falls into the scope of Section~\ref{sec:abelian-gp-bdl}; in particular, $B$ is a commutative algebra, and 
$\specB $ 
is
its spectrum, which comes with the map $\rho\colon \specB \to\agpd\z=\gpd\z$. As $\agpd $ is open in $\gpd $, $B$ is naturally a subalgebra of $A$, and normality of $\agpd $ implies that $B$ is regular in $A$. Furthermore, $B$ contains an approximate unit for $A$ because $\gpd\z\sset \agpd $. As $\agpd $ is also closed, the map $C_c (\agpd ,c) \to C_c (\gpd ,c)$ which extends functions by zero, extends to a faithful conditional expectation $\Phi\colon B\to A$; see \cite[Proposition~3.13]{DGNRW:Cartan}. Thus, Assumptions~\ref{assumption:Sigma}--\ref{assumption:mcA-maximal} make $B$ a Cartan subalgebra of $A$.

Let us explain why our last assumption on $\agpd $ is reasonable. It was shown in \cite[Theorem~3.1]{DGNRW:Cartan} that, in order to get \ref{assumption:mcA-maximal}, a sufficient assumption on $\agpd $ is that $(1)$ $\agpd $ is maximal among the Abelian subgroupoids of $\Iso(\gpd ) $ on which $c$ is symmetric, and additionally $(2)$ $\agpd $ is \emph{immediately centralizing} (see \cite[Definition~2]{DGNRW:Cartan}). A careful examination of the proof of \cite[Proposition 3.9]{DGNRW:Cartan} reveals that, instead of $(2)$, we may assume that for each $\eta\in \Iso(\gpd )$ with $u=r(\eta)=s(\eta)$, the set $\{ a \nu a\inv \, | \, a\in \agpd_{u}\}$ is either the singleton $\{\nu\}$ or infinite.

In the sections about to come, we will use the techniques we have developed so far to compute the Weyl groupoid $\mc{G}_{(A,B)}$ and the Weyl twist $\Sigma_{(A,B)}$ of the Cartan pair $(A,B)$. In particular, we will see in Theorems~\ref{thm:varphi} and~\ref{thm:psi} that there is a strong connection to a certain groupoid action of $\gpd  / \agpd $ on $\specB $. As such, it seems prudent to briefly state a few facts about the quotient groupoid $\gpd/\agpd.$
\begin{rmk}\label{rmk:mcGs-topol-properties}
    We let $q\colon \gpd  \to \gpd /\agpd =: \qgpd , \gamma \mapsto q(\gamma)=:\dot{\gamma},$ denote the quotient map. 
    Since $\agpd $ is a {\em wide} subgroupoid of $\gpd $ (i.e.\ $\agpd \sset\Iso(\gpd )$ is closed with $\agpd\z=\gpd\z$), openness of $\agpd $ and {\'e}taleness of $\gpd $ imply that $q$ is an open map. Since $\gpd $ is Hausdorff
    and $\agpd$ is closed in $\gpd$,
    this implies that $\qgpd $ is Hausdorff also. Furthermore, it follows from \cite[Corollary 2.13]{williams-groupoids} (taking $G = \agpd$ and $X = \gpd$) that $\qgpd $ is locally compact and second countable because $\gpd $ is.
    
    Lastly we point out that, if one is interested in groupoids $\agpd \sset\gpd $ with {\'e}tale quotient $\qgpd $ (as in \cite[Theorem 5.6]{IKRSW:2020}), then one must ask for $\agpd $ to be open in $\gpd ,$ as we have done.
\end{rmk}

We will now construct a continuous left action  $\tilde{\alpha}$ of the locally compact Hausdorff groupoid $\qgpd =\gpd /\agpd $ 
on the spectrum $\specB $ of $B=C_{r}^{*}(\agpd ,c)$, with the  moment map $\rho:\specB  \to \qgpd\z$ given by  $\rho|_{ \specB_{u}} = u $. In the following, we will write
\begin{equation*}
\qgpd \ast_{\rho} \specB  := \{ (\dot{\gamma}, \chi )\in \qgpd \times \specB \,|\, s_{\qgpd } (\dot{\gamma}) = \rho  (\chi )\}.
\end{equation*}

\begin{prop}\label{prop:tilde-alpha-is-action}
    Let $\rho\colon \specB \to\qgpd\z$ be given by $\rho|_{ \specB_{u}} = u. $  For $\gamma\in \gpd$, $(\dot{\gamma},\chi) \in \qgpd \ast_{\rho} \specB$, and $a \in \agpd_{r(\gamma)}$, define 
    \[ 	
    	\tilde{\alpha}_{\dot{\gamma}}(\chi) (a)
    	:=
    	\overline{c(\gamma, \gamma\inv)} \, c(\gamma\inv, a) \, c(\gamma\inv a, \gamma) \ \chi (\gamma\inv a \gamma)
    	.
    \]
    Then
 \begin{enumerate}[label= \textup{(\arabic*)}]
    \item\label{item:1} $\tilde{\alpha}_{\dot{\gamma}}(\chi)$ only depends on $\dot{\gamma}\in\qgpd = \gpd / \agpd$, not on $\gamma\in\gpd .$
    \item\label{item:2} $\tilde{\alpha}_{\dot{\gamma}}(\chi) \in \specB_{r(\gamma)} $ and $\tilde{\alpha}_{\dot{ \tau}\dot{\gamma}}(\chi) = \tilde{\alpha}_{\dot{ \tau}}\bigl(\tilde{\alpha}_{\dot{\gamma}}(\chi)\bigr)$.
    \item\label{item:3}  If $u\in\qgpd\z$, then $\tilde{\alpha}_{u}(\chi) = \chi$ for all $\chi\in \specB_{u}$.
    \item\label{item:continuous} The map     $\qgpd  \ast_{\rho} \specB  \to \specB $, $(\dot{\gamma}, \chi) \mapsto \tilde{\alpha}_{\dot{\gamma}} (\chi)$, is continuous.
\end{enumerate}
In other words, $\tilde{\alpha}$ is a continuous left action of~$\qgpd $ on~$\specB$ with moment map~$\rho$. 
\end{prop}

Before embarking on the proof, we
point out that the formula for $\tilde{\alpha}$ is not surprising. Indeed, if $\chi$ were defined  on all of $\gpd $ and satisfied Equation~\eqref{eq:chi-proj} (and, by extension, Equation~\eqref{eq:chi-a-inv}),  then  we would have
\begin{align*}
    \tilde{\alpha}_{\dot{\gamma}}(\chi) (a)
    &=
    \overline{c(\gamma, \gamma\inv)} \, c(\gamma\inv, a) \,  \chi (\gamma\inv a)\,\chi( \gamma)
    \\
    &=
    \overline{c(\gamma, \gamma\inv)} \, \chi (\gamma\inv)\, \chi( a)\,\chi( \gamma)
    =
    \chi(a).
\end{align*}

\begin{proof}
 One readily verifies \ref{item:1}--\ref{item:3} using the cocycle identity (Equation~\eqref{eq:cocycle}), the fact that $c$ is symmetric on the Abelian subgroupoid $\agpd $, that $c$ is normalized, and that $c(\gamma\inv, \gamma)=c(\gamma, \gamma\inv)$ for any $\gamma\in\gpd $ by \cite[Lemma 2.1]{DGNRW:Cartan}. 
 
 For \ref{item:continuous}, suppose that $(\dot{\gamma}_{i},\chi_{i}) \to (\dot{\gamma},\chi)$ in $\qgpd  \ast \specB $. We need to show 
    (see \cite[Proposition 3.3]{MRW96})
    that, if $ a_{i} \to a$ in $\agpd $ and $s( a_{i} ) = r(\gamma_{i} )$ for all $i$, then $\tilde{\alpha}_{\dot{\gamma}_{i}}(\chi_{i})( a_{i} ) \to \tilde{\alpha}_{\dot{\gamma}}(\chi)(a).$  Since the cocycle $c$ and both multiplication and inversion on $\gpd $ are continuous, we have 
    \[ \overline{c(\gamma_{i} , \gamma_{i}\inv) } c(\gamma\inv_{i},  a_{i}  ) c(\gamma_{i}\inv  a_{i}  , \gamma_{i} ) \to \overline{c(\gamma, \gamma\inv)} c(\gamma\inv, a ) c(\gamma\inv a , \gamma).\]
    Similarly, $\gamma_{i}\inv  a_{i} \gamma_{i} \to \gamma\inv a \gamma.$ Since the assumption $\chi_{i} \to \chi$ in $\specB $ implies in particular that $\chi_{i}(\gamma_{i}\inv  a_{i} \gamma_{i} ) \to \chi(\gamma\inv a\gamma)$, it follows that $\tilde{\alpha}_{\dot{\gamma_{i} }}(\chi_{i})( a_{i} ) \to \tilde{\alpha}_{\dot{\gamma}}(\chi)(a)$.
\end{proof}

 The action $\tilde{\alpha}$ of $\qgpd $ on $\specB $ allows us to endow the space $\qgpd \ast_{\rho} \specB $ with the structure of a topological groupoid. This so-called \emph{left action groupoid} is denoted $\qgpd  \ltimes \specB $, and we will show in Theorem~\ref{thm:varphi} that it is isomorphic to the Weyl groupoid $\mc{G}_{(A,B)}$.

Recall (see \cite[p.\ 3]{IK:gpd-actions-on-fractafolds}) that the elements $(\dot{\tau},\chi_{1})$, $(\dot{\gamma},\chi_{2}) \in \qgpd \ltimes \specB $ are composable if $\chi_{1}=\tilde{\alpha}_{\dot{\gamma}}( \chi_{2})$, and their product is given by
    \[
        (\dot{\tau},\chi_{1})(\dot{\gamma},\chi_{2}) = (\dot{\tau}\dot{\gamma},\chi_{2}).
    \]
    The inverse of an element $(\dot{\gamma},\chi)$ is $(\dot{\gamma}\inv,\tilde{\alpha}_{\dot{\gamma}}( \chi_{2}))$.  Therefore, 
    \[
        (\qgpd \ltimes \specB )\z = \left\{(\rho (\chi),\chi)|\chi\in\specB \right\},
    \]
    and so the unit space may be identified with $\specB $ using the projection onto the second factor, $(\rho (\chi),\chi) \mapsto \chi$. Note that $s(\dot{\gamma},\chi) = (s(\dot{\gamma}),\chi)$ and $r(\dot{\gamma},\chi) = (r(\dot{\gamma}),\tilde{\alpha}_{\dot{\gamma}}(\chi))$. The  topology of $\qgpd\ltimes\specB$ is inherited from $\qgpd \times \specB $; since $\qgpd $ and $\specB $ are locally compact, so is $\qgpd \ltimes \specB $, and the fact that $\qgpd$ is {\'e}tale and $\rho$ is continuous implies that $\qgpd \ltimes \specB$ is {\'e}tale.
    See \cite[p.\ 5]{mrw-equiv&isom} for more details.

Our next goal will be to  describe the relationship between the partial homeomorphisms $\alpha_n$
used to construct the Weyl groupoid $\mc{G}_{(A, B)}$ and the action $\tilde{\alpha}$ (see Proposition \ref{prop:from-alpha-to-tilde-alpha} below). We begin with a few preliminary results.

\begin{lemma}\label{lem:bisections-enough-for-Weyl-gpd}
    Every element of the Weyl groupoid associated to $(A, B)$ can be represented by some $(\alpha_n (x), n, x)$ where $n$ is an element of 
    \[
        N:=\left\{ f\in C_{c}(\gpd, c ) \,\vert\, \supp(f) \text{ is a bisection}\right\}.
    \]
\end{lemma}

\begin{proof}
    Since $\gpd $ is a second-countable locally compact Hausdorff {\'e}tale groupoid, it follows from \cite[Lemma 3.1.3]{AidansNotes} that $
        C_c (\gpd , c) = \mathrm{span} (N).$
    In particular, $\mathrm{span} (N)$ is dense in $A$. By \cite[Lemma 3.11]{DGNRW:Cartan}, $N$ is contained in the normalizer of $C_{c}(\agpd ,c)$, which implies $N\sset N(B)$. The claim then follows from \cite[Proposition\ 4.1]{DGNRW:Cartan}.
\end{proof}

\begin{lemma}
\label{lem:3-implies-2}
    Suppose $f_{i}\in N$ and $\chi\in \specB_{u}$.
    Let $\Phi\colon A \to B$ denote the faithful conditional expectation associated to the Cartan pair.
    \begin{enumerate}[label=\textup{(\arabic*)}]
        \item\label{item:Weyl-gpd-implies-points} If $[\alpha_{f_{1}}(\phi_{\chi}), f_{1}, \phi_{\chi}]=[\alpha_{f_{2}}(\phi_{\chi}), f_{2}, \phi_{\chi}]$, then $\phi_{\chi}(\Phi (f_{2}^{*}f_{1}))\neq 0.$
        \item\label{item:Weyl-twist-implies-points} If $\llbracket\alpha_{f_{1}}(\phi_{\chi}), f_{1}, \phi_{\chi}\rrbracket=\llbracket\alpha_{f_{2}}(\phi_{\chi}), f_{2}, \phi_{\chi}\rrbracket$, then $\phi_{\chi}(\Phi (f_{2}^{*}f_{1}))> 0.$
    \end{enumerate}
    Moreover, if $\dot{\gamma}_{i}\in q(\supp'(f_i))$ with $s(\dot{\gamma}_1)=s(\dot{\gamma}_{2})=u$, then either of the above assumptions implies $\dot{\gamma}_1=\dot{\gamma}_2$.
\end{lemma}

\begin{proof}
    We start by proving~\ref{item:Weyl-twist-implies-points}. By assumption, $\phi_\chi\in \dom(f_{1})\cap\dom(f_{2})$ and there exist $b_{1},b_{2}\in B$ with $\phi_\chi (b_{i})   >  0$ and such that $f_{1} b_{1} = f_{2} b_{2}$. In particular, since $f_{2}\in N\sset N(B)$, it follows that $f_{2}^{*}f_{1} b_{1} = f_{2}^{*} f_{2} b_{2}$ is an element of $B$. Since $\phi_\chi\in\dom(f_{2})$, we conclude that
    \[
        \phi_{\chi}(f_{2}^{*}f_{1} b_{1})
        =
        \phi_{\chi}(f_{2}^{*}f_{2} b_{2})
        =
        \phi_{\chi}(f_{2}^{*} f_{2})\, \phi_{\chi}( b_{2})
          >  0.
    \]
    As the conditional expectation $\Phi$ fixes $B$ and is $B$-linear, we get the equality in the following:
    \[
        \phi_{\chi}\bigl(\Phi( f_{2}^{*}f_{1})\bigr)\,\phi_\chi (b_{1})
        =
        \phi_{\chi}(f_{2}^{*}f_{1} b_{1})
          >  0.
    \]
    It follows that $\phi_{\chi}\bigl(\Phi( f_{2}^{*}f_{1})\bigr)  >  0$, as claimed.

    For \ref{item:Weyl-gpd-implies-points}, we use Proposition~\ref{prop:Weyl-gpd-similar-to-twist} to obtain $b_i\in B$ such that $\phi_\chi (b_i)\neq 0$ and $f_{1}b_{1}=f_{2}b_{2}$. The above proof now works \textit{mutatis mutandis}, replacing each instance of `$>0$' by `$\neq 0$'.

    Lastly, in either of the two cases, $f_{2}^{*}f_{1}$ is an element of $N\sset C_{c}(\gpd ,c)$. 
    As $\Phi(g) = g|_{\agpd}$ for $g \in C_c(\gpd, c)$ (cf.~\cite[Proposition 3.13]{DGNRW:Cartan}),
    $\Phi (f_{2}^{*}f_{1})\in C_{c}(\agpd ,c)$. Thus the  definition of $\phi_\chi$ yields
    \[
        0 \neq  \phi_{\chi}\bigl(\Phi( f_{2}^{*}f_{1})\bigr)
        =
        \sum_{a\in\agpd_{\rho(\chi)}}
        \chi(a)\,
        \Phi( f_{2}^{*}f_{1}) (a)
        = \sum_{a \in \agpd_{\rho(\chi)}} \chi(a) f_2^{*}f_1(a)
        ,
    \]
    and consequently
    \[
        \supp'(f_{2}^{*}f_{1}) \cap \agpd_{\rho(\chi)} \neq \emptyset.
    \]
   If $\dot \gamma_i \in q(\supp'(f_i))$ satisfy $s(\dot \gamma_1) = s(\dot \gamma_2) = u$, let
   $\gamma_{i}$ denote the representative of $\dot{\gamma}_i$ in $\supp' (f_{i})$. Since $\gamma_{2}\inv \gamma_{1}$ is then the unique element in $\supp'(f_{2}^{*}f_{1})$ with source $\rho(\chi)$, it follows that $\gamma_{2}\inv\gamma_{1}\in\agpd_{\rho(\chi)}$, i.e.\ $\dot{\gamma}_{1}=\dot{\gamma}_{2}$.
\end{proof}

\begin{prop}
\label{prop:from-alpha-to-tilde-alpha}
    Suppose $f\in N$ and $\chi\in \specB $.
    \begin{enumerate}[label=\textup{(\arabic*)}]
        \item\label{item:dom-f} $\rho(\chi) \in s(\supp'(f))$ if and only if $\phi_{\chi}\in\dom(f)$.
        \item\label{item:from-alpha-to-tilde-alpha} If $ \{\gamma\} = \supp'(f)\cap \gpd_{\rho(\chi)},$ so that $\chi\in\dom(\tilde{\alpha}_{\dot{\gamma}})$ and $\phi_\chi \in \dom(f)$, then we have
    $ \alpha_{f}(\phi_\chi) = \phi_{\tilde \alpha_{\dot \gamma}(\chi)}.$
    \end{enumerate}
\end{prop}

\begin{proof}
    Since $f$ is supported in a bisection, it follows from \cite[Lemma 3.11]{DGNRW:Cartan} that $f$ is a normalizer of $B$, so $\alpha_{f}$ exists and has domain $\dom(f) \sset \widehat{B}$.
It then follows from (a twisted variant of) \cite[Lemma 3.1.4]{AidansNotes} that
\begin{equation*}\label{eq:f*f-support} f^{*} f(\eta) = \begin{cases}
|f(\zeta)|^2, & \eta \in \gpd\z \text{ and } \supp'( f)\cap \gpd_{\eta} = \{\zeta\}\\
0, &
 \text{otherwise}.
\end{cases} \end{equation*}

Using that $\chi(\rho(\chi))=1$ and that $f^{*}f\in C_{c}(\agpd ,c)\sset B$, this implies
\begin{align}\label{eq:phi-chi-f*f}
\phi_{  \chi}(f^{*} f) &=\sum_{\eta \in \agpd_{\rho(\chi)}}   \chi(\eta)\, (f^{*}f)(\eta)
=
\begin{cases}
 |f(\gamma)|^2 & \text{ if } \supp'(f)\cap \gpd_{\rho(\chi)}   =   \{\gamma\},
 \\
 0 & \text{ if } \supp'(f)\cap \gpd_{\rho(\chi)}   =\emptyset.
 \end{cases}
\end{align}
This proves \ref{item:dom-f}.

 For \ref{item:from-alpha-to-tilde-alpha}, note first that $\rho(\chi)=s(\gamma)$, so $\chi\in \specB_{\rho(\chi)}$ is automatically an element of $\dom(\tilde{\alpha}_{\dot{\gamma}}) = \specB_{s(\gamma)}$. It remains to prove that $\alpha_{f}(\phi_\chi) = \phi_{\tilde \alpha_{\dot \gamma}(\chi)}.$ Recall that $\alpha_{f}$ is uniquely determined by satisfying
\begin{align*}
    x(f^{*} bf)& = \alpha_{f}(x)(b)\,  x(f^{*}f)
    \text{ for all } b \in C^{*}_r(\agpd , c) \text{ and }x \in \widehat{B}.
\end{align*}
Observe that for $b\in C_{c}(\agpd ,c)$ and $\eta \in \agpd $,
\begin{align*} 
f^{*} b f (\eta) &= \sum_{\xi \in\gpd^{s(\eta)}} \overline{f(\xi\inv \eta\inv) c( \eta \xi, \xi\inv\eta\inv)} b*f(\xi\inv) c(\eta \xi, \xi\inv) \\
&= \sum_{\xi\in\gpd^{s(\eta)}}
\sum_{\beta \in\gpd^{s(\xi\inv)}} \overline{f(\xi\inv \eta\inv) c(\eta \xi, \xi\inv\eta\inv)} b(\xi\inv \beta) f(\beta\inv)c(\eta \xi, \xi\inv )  c(\xi\inv\beta, \beta\inv)
 \\
&= \sum_{\xi, \beta\in\gpd^{s(\eta)}} \overline{f(\xi\inv \eta\inv) c( \eta \xi, \xi\inv\eta\inv)} b(\xi\inv \beta) f(\beta\inv)c(\eta \xi, \xi\inv )  c(\xi\inv\beta, \beta\inv).
\end{align*}

The factor $b(\xi\inv\beta)$ will only be nonzero if $r(\xi\inv \beta) = s(\xi\inv\beta)$, i.e.\ $r(\xi\inv) = s(\beta)$.  
In that case, since  $f$ is supported in a bisection and since $r(\beta\inv) = s(\beta)=r(\xi\inv)=r(\xi\inv\eta\inv)$, the only nonzero terms in the sum occur when $\beta = \eta \xi$,
so that 
\begin{align*} f^{*} b f(\eta) &= \sum_{\beta \in \gpd^{ r(\eta)}} |f(\beta\inv)|^2 b(\beta\inv \eta \beta) \overline{c(\beta, \beta\inv)} c(\beta, \beta\inv \eta) c(\beta\inv \eta \beta, \beta\inv)
\\
&=\sum_{\zeta \in \gpd_{r(\eta)}} |f(\zeta )|^2 b(\zeta  \eta\zeta\inv) \overline{c(\zeta\inv,\zeta)} c(\zeta\inv,\zeta  \eta) c(\zeta  \eta\zeta\inv,\zeta ) .
\end{align*}
In particular, if $\eta\in \agpd_{\rho(\chi)}^{\rho(\chi)}$ -- so that $r(\eta)=\rho(\chi)$ -- then the assumption $\supp'(f)\cap \gpd_{\rho(\chi)}
    =
    \{\gamma\}$ shows that
\begin{align*}
f^{*} b f(\eta) &= |f(\gamma )|^2 b(\gamma  \eta \gamma\inv) \overline{c( \gamma\inv, \gamma)} c(\gamma\inv, \gamma  \eta) c(\gamma  \eta \gamma\inv, \gamma ) .
\end{align*}

Therefore, 
 for $x=\phi_{\chi}$  and $b\in C_{c}(\agpd ,c)\sset B$,
\begin{align}\label{eq:phichi-f*bf}
     \phi_{  \chi}(f^{*} b f) =& \sum_{\eta \in \agpd_{\rho(\chi)}}   \chi(\eta)\, (f^{*} bf)(\eta)\notag
     \\
     =&
    |f(\gamma )|^2 
    \sum_{\eta \in \agpd_{\rho(\chi)}}   \chi(\eta)\, b(\gamma \eta \gamma\inv) \overline{c( \gamma\inv, \gamma)} c(\gamma\inv, \gamma  \eta) c(\gamma \eta \gamma\inv, \gamma).
\end{align}
By Equation~\eqref{eq:phi-chi-f*f} and by definition of $\alpha_{f}$, we have
\[
    \alpha_{f}(\phi_{  \chi})(b) \, \abs{f(\gamma)}^2
    =
    \alpha_{f}(\phi_{  \chi})(b)\,  \phi_{\chi}(f^{*}f)
    =
    \phi_{\chi}(f^{*} bf),
\]
and since $f(\gamma) \neq 0$ as $\gamma\in\supp' (f)$, Equation~\eqref{eq:phichi-f*bf} allows us to conclude that 
\begin{align*}
    \alpha_{f}(\phi_{  \chi})(b)
    &=
    \sum_{\eta \in \agpd_{\rho(\chi)}}   \chi(\eta)\, b(\gamma \eta \gamma\inv) \overline{c( \gamma\inv, \gamma)} c(\gamma\inv, \gamma  \eta) c(\gamma \eta \gamma\inv, \gamma)
    \\
    & = \sum_{a \in \agpd_{r(\gamma)}}
\chi (\gamma\inv a \gamma)\, \overline{c(\gamma, \gamma\inv}) c(\gamma\inv, a \gamma ) c(a, \gamma)b(a)
    \\
    &= \sum_{a \in \agpd_{r(\gamma)}} \chi (\gamma\inv a \gamma)\, \overline{c(\gamma, \gamma\inv)} c(\gamma\inv a, \gamma ) c(\gamma\inv, a)b(a) =  \phi_{\tilde \alpha_{\dot \gamma}}(\chi) (b).
\end{align*}
To obtain the second equality, we invoked the fact that $c(\gamma, \gamma\inv ) = c(\gamma\inv , \gamma)$ for any $\gamma \in \gpd$ (see \cite[Lemma 2.1]{DGNRW:Cartan}).
It follows that, as desired, $\alpha_{f}(\phi_\chi) (b) = \phi_{\tilde \alpha_{\dot \gamma}(\chi)} (b)$ for all $b$ in the dense subalgebra $C_{c}(\agpd )$ and thus on all of $B$.
\end{proof}

The fact that $
        \alpha_{f}(\phi_\chi) = \phi_{\tilde \alpha_{\dot \gamma}(\chi)}$ whenever $ \{\gamma\} = \supp'(f)\cap \gpd_{\rho(\chi)}$ for $f\in N$ and $\chi\in \specB$, shows that there is an intimate connection between the partial action $\alpha$ of $N(B)$ on $\specB$ and the action $\tilde{\alpha}$ of $\qgpd$ on $\specB$. In order to describe this connection, we first need to better understand equality in the Weyl groupoid $\mc{G}_{(A,B)}$.

\begin{prop}
\label{prop:Weyl-gpd-equality-equivalences}
    Suppose $f_{i}\in N$ and let $X^{i}:= \mathrm{supp}'(f_{i})\sset\gpd $ and $\chi\in \specB $. 
    Then the following are equivalent:
    \begin{enumerate}[label=\textup{(\arabic*)}]
        \item\label{item:same-in-unitspace} There exists an open neighborhood $U$ of $\rho(\chi)$ in $s(X^{1})\cap s(X^{2}) \sset \gpd\z$ such that $q(X^{1}_{u}) = q(X^{2}_{u})$ for all $u\in U$.
        \item\label{item:same-on-rho-chi} $q(X^{1}_{\rho(\chi)}) = q(X^{2}_{\rho(\chi)})$.
        \item\label{item:same-in-Weyl} $\phi_\chi\in\dom(f_1) \cap \dom(f_2)$ and $[\alpha_{f_{1}}(\phi_{\chi}), f_{1}, \phi_{\chi}]=[\alpha_{f_{2}}(\phi_{\chi}), f_{2}, \phi_{\chi}]$.
    \end{enumerate}
\end{prop}

In the above, we wrote $X^{i}_{u}$ for the singleton-set $X^{i}\cap \gpd_{u}$, and $q\colon \gpd \to \qgpd $ for the quotient map.
\begin{proof}
Note that  \ref{item:same-in-Weyl}$\implies$\ref{item:same-on-rho-chi} is the second assertion of Lemma~\ref{lem:3-implies-2}\ref{item:Weyl-gpd-implies-points}, so it suffices to prove \ref{item:same-on-rho-chi}$\implies$\ref{item:same-in-unitspace}$\implies$\ref{item:same-in-Weyl}. 

   Assume that \ref{item:same-on-rho-chi} holds. By
   \cite[Proposition 2.2.4]{Paterson:book},
   $(X^{1})\inv\cdot X^{2}$ is an    open   bisection in $\gpd $.    Setting   $\{\gamma_{i} \}=X^{i}_{\rho(\chi)}$,    by assumption $\dot{\gamma}_{1} = \dot{\gamma}_{2}$; in particular, $\gamma_{1}\inv \gamma_{2} \in [(X^{1})\inv\cdot X^{2} ] \cap \agpd $. Since $\agpd $ is open and $\gpd $ is {\'e}tale, the set $U:=s([(X^{1})\inv\cdot X^{2} ] \cap \agpd )$ is  
   an open subset of $\gpd\z$ which contains $s(\gamma_{1}\inv\gamma_{2}) = s(\gamma_{2})=\rho(\chi)$. Furthermore, $U$ is contained in $s(X^{1})\cap s(X^{2})$: Any $u\in U$ can be written as $u=s(\gamma_{u}\inv  \tau_{u})$ for $\gamma_{u}\in X^{1}$ and $ \tau_{u} \in X^{2}$ such that $\gamma_{u}\inv  \tau_{u}\in\agpd $. In particular,
    \begin{alignat*}{2}
        s( \tau_{u}) &= s(\gamma_{u}\inv  \tau_{u}) = u &&\text{by assumption, and}
        \\
        s(\gamma_{u})&
        = r(\gamma_{u}\inv  \tau_{u})
        \overset{(*)}{=}s(\gamma_{u}\inv  \tau_{u})=u,\quad&&\text{where $(*)$ follows from }\agpd  \sset \Iso(\gpd ).
    \end{alignat*}
    This shows that $\{\gamma_{u}\} = X^1_{u}$, $\{ \tau_{u}\} = X^2_{u}$, and $U\sset s(X^{1})\cap s(X^{2})$. Moreover, it follows from  $\gamma_{u}\inv \tau_{u}\in\agpd $ that $q(X^1_{u})=q(X^2_{u})$. This proves~\ref{item:same-in-unitspace}.
    
    Next, we will show \ref{item:same-in-unitspace}$\implies$\ref{item:same-in-Weyl}.
    Pick any $u\in U \sset s(X^{1})\cap s(X^{2})$, and let $
    \gamma^{i}_u 
    $
    denote the unique element in $X^{i}_{u} = s\inv (u) \cap X^{i}$. Note that, by our assumption on $U$, we have $
    \dot{\gamma}^{1}_u=\dot{\gamma}^2_{u}
    $. Using Proposition~\ref{prop:from-alpha-to-tilde-alpha}\ref{item:from-alpha-to-tilde-alpha} for $(*)$ in the following, we thus see that,  for any $\nu\in\specB_{u}$,
    \[
        \alpha_{f_{1}} (\phi_{\nu})
        \overset{(*)}{=}
        \phi_{\tilde{\alpha}_{\dot{\gamma}^{1}_u}(\nu)}
        =
        \phi_{\tilde{\alpha}_{\dot{\gamma}^{2}_u}(\nu)}
        \overset{(*)}{=}
        \alpha_{f_{2}} (\phi_{\nu})
        .
    \]
    Since $u$ was arbitrary, this shows that $\alpha_{f_{1}}$ and $\alpha_{f_{2}}$ coincide on all of $\phi (\rho\inv (U)).$ This set contains $\phi_\chi$ by the assumption that $\rho(\chi)\in U$ and it is open in $\widehat{B}$ since $\phi$ is a homeomorphism, $\rho$ is continuous, and $U$ is open in $\gpd\z$. This proves \ref{item:same-in-Weyl}.
\end{proof}

\begin{thm}\label{thm:varphi}
There is an isomorphism $\varphi$ of topological groupoids $\qgpd  \ltimes \specB  \to \mc{G}_{(A,B)}$ given by
\begin{equation} 
\label{eq:defn-of-varphi} \varphi(\dot{\gamma}, \chi) := [\phi_{\tilde{\alpha}_{\dot{\gamma}}(\chi)},  f , \phi_\chi] =[\alpha_{f}(\phi_\chi),  f , \phi_\chi], 
\end{equation}
where $f\in C_{c}(\gpd ,c)$ is any function supported on a bisection such that $\dot{\gamma}\in q(\supp'(f))$.
\end{thm}

We point out here that this result is a significant strengthening of \cite[Theorem 5.2]{DGNRW:Cartan}. Not only is Theorem \ref{thm:varphi}  true for {\'e}tale groupoids (not just discrete groups), but we also do not need to assume that $\tilde{\alpha}$ is topologically free. Instead, as our theorem suggests, this simply follows from $\mc{G}_{(A,B)}$ being topologically principal. Moreover, Theorem \ref{thm:varphi} applies in the setting of \cite[Section 3]{IKRSW:2020} if one assumes the groupoids involved to be {\'e}tale, as discussed in the following remark.

\begin{rmk}
\label{rmk:comparison-with-IKRSW} 
When the 2-cocycle $c$ on the {\'e}tale groupoid $\gpd $ is trivial, $\specB$ is precisely the dual bundle $\widehat{\agpd }$ used in \cite[Section 3]{IKRSW:2020}. Moreover, the right action of $\gpd /\agpd $ on $\widehat{\agpd }$ given at the bottom of \cite[page 14]{IKRSW:2020},
\[ \chi \cdot \dot \gamma (a) = \chi(\gamma a \gamma\inv),\]
is precisely the left action $\tilde \alpha_{\dot \gamma\inv}(\phi_\chi)$ of $\dot \gamma\inv \in \gpd /\agpd $ on $\phi_\chi$ in this case. In other words, Theorem \ref{thm:varphi} establishes that the groupoid $\widehat{\agpd } \rtimes \qgpd $ of \cite[Theorem 3.4]{IKRSW:2020} is indeed the Weyl groupoid, if $\gpd$ is {\'e}tale and $C^{*}_r(\agpd )$ is Cartan in $C^{*}_r(\gpd )$.
\end{rmk}

We will use the rest of this section to prove Theorem \ref{thm:varphi} through a series of lemmata.

\begin{lemma}\label{lem:varphi-welldefined}\label{lem:varphi-homom}
    The map $ \varphi \colon  \qgpd   \ltimes \specB  \to \mc{G}_{(A,B)}$ 
    of Equation \eqref{eq:defn-of-varphi}
    is a 
    well-defined groupoid homomorphism.
\end{lemma}
\begin{proof}
Let $(\dot \gamma, \chi) \in \qgpd\ltimes \specB$ and $f \in N$ satisfy $\dot \gamma \in q(\supp'(f)).$
This assumption
implies $\rho(\chi)=s(\gamma)\in s(\supp'(f))$, which guarantees that $\phi_\chi \in \dom (f)$ by Proposition~\ref{prop:from-alpha-to-tilde-alpha}\ref{item:dom-f}, so that $[\phi_{\tilde{\alpha}_{\dot{\gamma}}(\chi)},  f , \phi_\chi]$ is indeed an element of $\mc{G}_{(A,B)}$. Moreover, this element is independent of the choice of $f$ by Proposition~\ref{prop:Weyl-gpd-equality-equivalences}, \ref{item:same-on-rho-chi}$\implies$\ref{item:same-in-Weyl}. 
In other words, $\varphi$ is well-defined.

    Next suppose $\bigl((\dot{ \tau}, \chi'), (\dot{\gamma},\chi)\bigr)$ is a composable pair in $\qgpd  \ltimes \specB $,  i.e.\ $\chi' = \tilde{\alpha}_{\dot{\gamma}}(\chi)$. 
    It follows that $\varphi$ takes this composable pair to a composable pair:
    \[s (\varphi(\dot\tau, \chi')) = \phi_{\chi'} = \phi_{\tilde{\alpha}_{\dot \gamma}(\chi)} =r(\varphi(\dot\gamma, \chi)).\]
  Moreover, if $g,h\in N$ with $\dot{ \tau} \in q(\supp'(g))$ and $\dot{\gamma}\in q(\supp' (h))$, then
    \[\varphi(\dot{ \tau}, \chi') \varphi (\dot{\gamma},\chi) = [\phi_{\tilde{\alpha}_{\dot{ \tau}}(\chi')}, g, \phi_{\chi'}] [\phi_{\tilde{\alpha}_{\dot{\gamma}}(\chi)}, h, \phi_{\chi}] = [\phi_{\tilde{\alpha}_{\dot{ \tau}}( \tilde{\alpha}_{\dot{\gamma}} (\chi))},
    g h,
    \phi_{\chi} ],
   \] 
     which equals $ [\phi_{\tilde \alpha_{\dot \tau \dot \gamma}(\chi)},
    g h%
    , \phi_\chi]$ since $\tilde \alpha$ is an action.
    On the other hand,
     \[ \varphi\left((\dot{ \tau}, \chi') (\dot{\gamma},\chi)\right) = \varphi(q( \tau\gamma), \chi) = [\phi_{\tilde{\alpha}_{q( \tau\gamma)}(\chi)}, f, \phi_{\chi}],
    \]
    where $f\in N$ with $q( \tau\gamma)\in q(\supp'(f))$.
    
    In order to show that $[\phi_{\tilde{\alpha}_{q( \tau\gamma )}(\chi)}, f, \phi_{\chi}] =
    [\phi_{\tilde{\alpha}_{\dot{ \tau}{\dot{\gamma}}}(\chi)},
    g h%
    , \phi_{\chi} ]
    ,$
    it suffices to show that, like $f$, $
    g h%
    $ is an element of $N$ with $q( \tau\gamma) \in q(\supp'(
    g h%
    ))$.
    
    Since $g, h \in N $, 
    $
    g h%
    $ is supported on 
   the bisection
   $\supp(g)\cdot \supp(h)$ (cf.~\cite[Lemma 3.1.4]{AidansNotes} in the untwisted case). Lastly, if $ \tau \in \supp'(g)$ and $\gamma \in \supp'(h)$ are representatives of $\dot{ \tau}$ and $\dot{\gamma}$ respectively, then
    \[0\neq  g( \tau)h(\gamma)c(\tau, \gamma) =(
    g h %
    )( \tau \gamma) ,\]
    so $q( \tau\gamma)\in q (\supp'(
    g h %
    ))$.
    Therefore, the fact that $\varphi$ is well-defined implies that $[\phi_{\tilde{\alpha}_{q( \tau\gamma)}(\chi)}, f, \phi_{\chi}] = [\phi_{\tilde{\alpha}_{\dot{ \tau}}( \tilde{\alpha}_{\dot{\gamma}}(\chi))}, 
    g h %
    , \phi_{\chi} ]$.
\end{proof}
    
\begin{lemma}\label{lem:varphi-onto}\label{lem:varphi-injective-wo-principalness}
    The  
    groupoid homomorphism   
    $\varphi\colon  \qgpd  \ltimes \specB  \to \mc{G}_{(A,B)}$ is a
    bijection.
\end{lemma}    
    
\begin{proof}
   By Lemma~\ref{lem:bisections-enough-for-Weyl-gpd}, we know that every element of $\mc{G}_{(A,B)}$ is of the form $
        [\alpha_{f} (x), f, x]
    $
    for $f \in N$
    and $x\in \dom(f) \sset \widehat{B}$. By Lemma~\ref{lem:phi-surjective}, we may write $x=\phi_{\chi}$ for a unique $\chi\in\specB $. Since $
    \phi_{\chi} (f^{*}f)>0$ by assumption, Proposition~\ref{prop:from-alpha-to-tilde-alpha}\ref{item:dom-f} shows that $\rho(\chi) \in s(\supp'(f))$, i.e.\ $ \supp'(f)\cap \gpd_{\rho(\chi)}=  \{\gamma\} $ for some $\gamma$. In particular,  
    $(\dot{\gamma},\chi)\in \qgpd  \ltimes \specB  .$ Moreover, 
    the fact that $\varphi$ is well-defined means that
    \[
        \varphi (\dot{\gamma},\chi) 
        =
        [\phi_{\tilde{\alpha}_{\dot{\gamma}}(\chi)},
        f
        , \phi_\chi].
    \]
    Proposition~\ref{prop:from-alpha-to-tilde-alpha}\ref{item:from-alpha-to-tilde-alpha}
    now implies that $\varphi (\dot{\gamma},\chi)  = [\alpha_{f} (x), f, x]$, so $\varphi$ is surjective.

  For injectivity, assume that $\varphi (\dot{\gamma}_{1},\chi)=\varphi(\dot{\gamma}_{2},\chi')$, i.e.\
  \[
    [\alpha_{f_{1}}(\phi_\chi), f_{1}, \phi_\chi] = [\alpha_{f_{2}}(\phi_{\chi'}), f_{2}, \phi_{\chi'}],
    \text{ where } f_{i}\in N \text{ have } \dot{\gamma_{i} }\in\supp'(f_{i}).
  \]
  This immediately forces $\chi=\chi'$ by definition of the Weyl groupoid and injectivity of 
  $\phi$
  (Lemma~\ref{lem:phionetoone}). In particular, $s(\dot{\gamma}_{1})=\rho(\chi)=s(\dot{\gamma}_{2})$.

  Proposition~\ref{prop:Weyl-gpd-equality-equivalences}, \ref{item:same-in-Weyl}$\implies$\ref{item:same-on-rho-chi}, tells us that  if $X^{i}:=\supp'(f_{i})$, then $q(X^{1}_{\rho(\chi)})=q(X^{2}_{\rho(\chi)})$. By choice of $f_{i}$, $X^{i}_{\rho(\chi)}=\{\gamma_{i} \}$, 
  so
  $\dot{\gamma}_{1}=\dot{\gamma}_{2}$ 
  and $(\dot \gamma_1, \chi) = (\dot \gamma_2, \chi')$.
\end{proof}

\begin{lemma}\label{lem:varphi-cts}\label{lem:varphi-open}
    The  
    bijective groupoid homomorphism
    $\varphi\colon  \qgpd  \ltimes \specB  \to \mc{G}_{(A,B)}$ is a
    homeomorphism.
\end{lemma}

\begin{proof}
    To see continuity of $\varphi$, 
    suppose the net $(\dot{\gamma}_{i}, \chi_{i})_{i \in \Lambda}$ converges to $(\dot{\gamma}, \chi)$ in $\qgpd  \ltimes \specB $. Let $\varphi(\dot{\gamma}_{i}, \chi_{i}) = [\phi_{\tilde{\alpha}_{\dot{\gamma}_{i}}(\chi_{i})}, f_{i}, \phi_{\chi_{i}}]$
    and $\varphi(\dot \gamma, \chi) = [\phi_{\tilde \alpha_\gamma(\chi)}, f, \phi_\chi],$
    where $f_{i}, f \in C_c(\gpd ,c)$ are supported on  bisections such that $\dot{\gamma}_{i} \in q(\supp'(f_{i}))$ 
    and $\dot \gamma \in q(\supp'(f))$.
    
    Let $U$ be an open neighborhood of $\phi_{\tilde{\alpha}_{\dot{\gamma}}(\chi)}$ in $\widehat{B}$ and $V$ be an open neighborhood of $\phi_\chi$ in $\widehat{B}$, so that
    \begin{equation}\label{eq:def:mcU}
         \mc{U}(U,f,V):= \{[\alpha_{f}(\phi_\nu), f, \phi_\nu] \in \mc{G}_{(A,B)} \mid \alpha_{f}(\phi_\nu) \in U, \phi_\nu \in V\}
    \end{equation}
    is a basic open neighborhood
    of $[\alpha_{f}(\phi_\chi), f, \phi_\chi]$ (see \cite[p.\ 36]{renault-cartan}). We must show that
    there exists $N \in \Lambda$ such that if $i \geq N$, then $[\phi_{\tilde{\alpha}_{\dot{\gamma}_{i}}(\chi_{i})}, f_{i}, \phi_{\chi_{i}}]$ lies in $\mc{U}(U,f,V)$.
    Since $q(\supp'(f))$ is an open neighborhood of $\dot{\gamma}$ and  $\dot{\gamma}_{i}$ converges to $\dot{\gamma}$,  there exists $N_1 \in \Lambda$ such that if $i \geq N_1$ then $\dot\gamma_{i} \in q(\supp'(f))$. By the well-definedness of $\varphi$ (Lemma \ref{lem:varphi-welldefined}), we may assume that $f_{i} = f$ for $i \geq N_1$. Moreover, since $\chi_{i} $ converges to $\chi$ by assumption, Lemma \ref{lem:phi-continuous} implies there exists $N_2 \in \Lambda$ such that $\phi_{\chi_{i}} \in V$ for all $i \geq N_2$. Lastly, since $(\dot{\gamma}_{i}, \chi_{i})$ converges to $(\dot{\gamma}, \chi)$, continuity of $(\dot{\tau},\nu)\mapsto \tilde{\alpha}_{\dot{\tau}}(\nu)$ (Proposition~\ref{prop:tilde-alpha-is-action}\ref{item:continuous}) and of $\nu\mapsto\phi_\nu$ (Lemma~\ref{lem:phi-continuous}) imply that there exists $N_3 \in \Lambda$ such that $\phi_{\tilde{\alpha}_{\dot{\gamma}_{i}}(\chi_{i})} \in U$ for all $i \geq N_3$. Therefore, if 
    $i$ is greater than each of $N_{1},N_{2},N_{3}$, then $[\phi_{\tilde{\alpha}_{\dot{\gamma}_{i}}(\chi_{i})}, f, \phi_{\chi_{i}}]= [\phi_{\tilde \alpha_{\dot \gamma_i}(\chi_i)}, f_i, \phi_{\chi_i}]$ lies in 
    $\mc{U}(U,f,V)$.

    To see that $\varphi\inv$ is continuous, 
    suppose $(\dot{\gamma}_{i}, \chi_{i})_{i \in \Lambda}$ is a net in $\qgpd  \ltimes \specB $ such that $\varphi(\dot{\gamma}_{i}, \chi_{i})$ converges to some element $\WeylGpdElement$ in $\mc{G}_{(A,B)}$. If we write $\WeylGpdElement =[\alpha_{f}(\phi_\chi), f, \phi_\chi]$ for some $f\in N$ with $\phi_\chi\in\dom(f)$, then a basic open neighborhood around $\WeylGpdElement$ is of the form $\mc{U}(\widehat{B},f,V)$ (see Equation~\eqref{eq:def:mcU}), where $V\sset \dom(f)$ is some open neighborhood of $\phi_\chi$. Since $\varphi(\dot{\gamma}_{i}, \chi_{i})\to \WeylGpdElement$, we know that for any fixed $V$, there exists $N \in \Lambda$ such that if $i \geq N$ then $\varphi(\dot{\gamma}_{i}, \chi_{i})$ lies in $\mc{U}(\widehat{B},f,V)$. This means in particular that $\phi_{\chi_{i}}\in V$. Since $V$ was arbitrary, this proves that $\phi_{\chi_{i}}\to \phi_{\chi}$.
    The fact 
    that the map $\phi$ is open (Lemma~\ref{lem:phi-open})
    implies that $\chi_{i}$ converges to $\chi$.
    Since $\phi_\chi, \phi_{\chi_{i}}\in\dom(f)$, it follows from Proposition~\ref{prop:from-alpha-to-tilde-alpha}\ref{item:dom-f} that $\rho(\chi), \rho(\chi_{i})\in s(\supp'(f)),$ i.e.\ there exist (unique) $\tau,\tau_{i}\in\supp'(f)$ with $s(\tau)=\rho(\chi)$ and $s(\tau_{i})=\rho(\chi_{i}).$ The definition of $\varphi$ thus yields
    \[
        \varphi(\dot{\tau}, \chi) = [\alpha_{f}(\phi_{\chi}), f, \phi_{\chi}] = \WeylGpdElement
        \quad\text{and}\quad
        \varphi(\dot{\tau}_{i}, \chi_{i}) = [\alpha_{f}(\phi_{\chi_{i}}), f, \phi_{\chi_{i}}].
    \]
    On the other hand, since $\varphi(\dot{\gamma}_{i},\chi_{i})\in \mc{U}(\widehat{B},f,V)$, we also know that
    \[
        \varphi(\dot{\gamma}_{i},\chi_{i}) = [\alpha_{f}(\phi_{\chi_{i}}), f, \phi_{\chi_{i}}].
    \]
    Injectivity of $\varphi$ now implies $\dot{\gamma}_{i} = \dot{\tau}_{i}$. We will show that $\tau_{i}\to\tau$, so that in particular $\dot{\gamma}_{i}=\dot{\tau}_{i} \to \dot{\tau}$.
    
    Let $U$ be any open neighborhood around $\tau$ contained in $\supp'(f)$. Then, since $\gpd $ is {\'e}tale, $s(U)$ is an open neighborhood around $s(\tau)=\rho(\chi).$ As $\chi_{i}\to\chi$ and $\rho$ is continuous, there exists $i_0$ so that for all $i\geq i_0$, we have $s(\tau_{i})=\rho(\chi_{i})\in s(U)$. Since $\tau_{i}\in \supp'(f)$ and $U = s\inv(s(U)) \cap \supp'(f) 
    $, this means $\tau_{i} \in U$ for $i\geq i_0$. As $U$ was an arbitrary neighborhood around $\tau$, we conclude that $\tau_{i}$ converges to $\tau$.

    All in all, we have shown that $(\dot{\gamma}_{i},\chi_{i})$ converges to $(\dot{\tau},\chi)$, the pre-image of $\WeylGpdElement$ under $\varphi$. This completes the proof 
    of the lemma, and also of Theorem~\ref{thm:varphi}.
\end{proof}

\pagebreak[3]\section{A cocycle on the Weyl groupoid}\label{sec:cocycle}

 The standing assumptions for this section can be found on page~\pageref{assumption:Sigma}. In this section, we show that every  (continuous) section $\qsec $ of the quotient map $q$ induces a (continuous) $\mbb{T}$-valued $2$-cocycle $C^\qsec $ on the Weyl groupoid. Though the formula for $C^\qsec $ is complicated, it is reminiscent of \cite[Lemma 5.3]{renault-cartan} and will be exactly what we require to prove in Theorem~\ref{thm:psi} that $(\qgpd \ltimes\specB )\times_{C^\qsec } \mbb{T}$ and the Weyl twist $\Sigma_{(A,B)}$ are isomorphic  (even as topological groupoids, when $\qsec$ is continuous).
 
\begin{prop}\label{prop:def:cocyle-for-Weyl}
  Let $\qsec \colon\gpd/\agpd = \qgpd \to\gpd $ be a section for  $q\colon\gpd \to\qgpd $ and for each $\dot{\gamma}\in \qgpd $, choose one $f_{\dot{\gamma}}\in N$ such that $f_{\dot{\gamma}}(\qsec (\dot{\gamma})) >0$. The following defines a $\mbb{T}$-valued 2-cocycle on $ \qgpd  \ltimes \specB  $:
    \[
        C^{\qsec }\bigl( (\dot{\gamma}_{1},\tilde{\alpha}_{\dot{\gamma}_{2}}(\chi)), (\dot{\gamma}_{2},\chi)\bigr) :=
        \frac{
            \phi_\chi \bigl(\Phi( f_{q(\gamma_{1}\gamma_2)}^{*}  f_{\dot{\gamma}_{1}}  f_{\dot{\gamma}_{2}})\bigr)
        }{
            \abs{\phi_\chi \bigl(\Phi( f_{q(\gamma_{1}\gamma_2)}^{*}  f_{\dot{\gamma}_{1}}  f_{\dot{\gamma}_{2}})\bigr)}
        },
    \]
    where 
    $\Phi\colon A\to B$ is the conditional expectation. If $\qsec $ is continuous, then $C^{\qsec }$ is continuous.
\end{prop}

We will deal with the algebraic and topological
aspects of the proof
separately.

\begin{lemma}
    The map $C^{\qsec }$ from Proposition~\ref{prop:def:cocyle-for-Weyl} is 
    $\mbb{T}$-valued
    and satisfies the cocycle condition
    of Equation \eqref{eq:cocycle}.
\end{lemma}
\begin{proof}
    Let $C:= C^{\qsec }.$    First note that, if 
    we let
    $f:=f^{*}_{q(\gamma_1 \gamma_2)} f_{\dot{\gamma}_1} f_{\dot{\gamma}_2} \in N$, then
  writing $\gamma_i = \qsec (\dot{\gamma}_i),$ there exists $a \in \agpd_{s(\gamma_2)} $ with $\qsec (q(\gamma_1\gamma_2)) a = \gamma_1 \gamma_2.$ 
  Using the fact that each $f_\eta$ is supported on a bisection, one  computes that
  \[ f(a) = \overline{f_{q(\gamma_1\gamma_2)}(\gamma_1 \gamma_2) c(\gamma_1 \gamma_2 a\inv , a \gamma_2^{-1} \gamma_1^{-1})} f_1(\gamma_1) f_2 (\gamma_2) c(\gamma_1, \gamma_2) c(a \gamma_2^{-1} \gamma_1^{-1}, \gamma_1 \gamma_2)\]
  is nonzero,
  and hence  
$\{a\}= \supp'(f)\cap \agpd_{s(\gamma_2)}$.
  Since $\Phi$ is, on $C_c(\gpd )$, given by restriction 
  to $\agpd $, we see that $a\in \supp' (\Phi(f))$ also. 
  The fact that $\agpd \subseteq \gpd$ is open implies that
  $\Phi(f)\in N\cap B$.  As  $\rho(\chi)=s(\gamma_{2})=s(a)$, it follows from Proposition~\ref{prop:from-alpha-to-tilde-alpha}\ref{item:dom-f} that $\phi_\chi\in \dom (\Phi(f)),$ i.e.\
    \[
        \abs{\Omega(\Phi(f))}^2 (\phi_\chi)
        =
        \phi_\chi \bigl(\Phi(f)^{*} \Phi(f)\bigr)
        >0,
    \]
    which implies $\phi_\chi\bigl(\Phi(f)\bigr) =\Omega(\Phi(f)) (\phi_\chi) \neq 0$ by the $C^*$-identity. This shows that $C$ is 
    $\mbb{T}$-valued.
    
    For the cocycle identity (Equation~\eqref{eq:cocycle}), let
    \[
        h_{3}:=(\dot{\gamma}_3, \chi),
        \;
        h_{2}:=(\dot{\gamma}_2, \tilde{\alpha}_{\dot{\gamma}_3}(\chi)),
        \;
        h_{1}:=(\dot{\gamma}_1, \tilde{\alpha}_{\dot{\gamma}_2\dot{\gamma}_3}(\chi)).
    \]
    We need to check that
    \[
        C(h_{1}h_{2},h_{3})\ C(h_{1},h_{2})
        \overset{!}{=}
        C(h_{1},h_{2}h_{3})\ C(h_{2},h_{3}).
    \]
    In the following computation, we write $\propto$
    to denote that the left-hand side is a {\em positive} multiple of the right-hand side. 
    If we let $\tilde{C}$ denote the numerator of $C$, then our claim becomes
    \begin{equation}\label{eq:cocycle-claim}
        \tilde{C}(h_{1}h_{2},h_{3})\ \tilde{C}(h_{1},h_{2})\
        \phi_\chi ( f_{\dot{\gamma}_{3}}^{*}f_{\dot{\gamma}_{3}}) \
        \overset{!}{\propto} \
        \tilde{C}(h_{1},h_{2}h_{3})\ \tilde{C}(h_{2},h_{3}),
    \end{equation}
    where we threw in the positive number $\phi_\chi ( f_{\dot{\gamma}_{3}}^{*}f_{\dot{\gamma}_{3}})$ to eventually make up for the following annoyance: out of the four $\tilde{C}$-factors, $\tilde{C}(h_{1},h_{2})$ is {\em not} given by evaluating at $\phi_\chi$ but at $\tilde{\alpha}_{\dot{\gamma}_3}(\chi)=\alpha_{f_{\dot{\gamma}_3}} (\phi_\chi).$
    
    To protect our sanity, let us denote
    \[
        f_{i}:= f_{\dot{\gamma}_{i}},\quad
        f_{ij}:= f_{q(\gamma_{i}\gamma_j)},\quad
        f_{123}:= f_{q(\gamma_1\gamma_2\gamma_3)},
        \quad
        i,j\in\{1,2,3\}
    \]
    and 
    \[ \gamma_i := \qsec (\dot{\gamma_i}), \quad 
    \gamma_{ij} := \qsec (q(\gamma_i \gamma_j)), \quad
    \gamma_{123} := \qsec (q(\gamma_1 \gamma_2 \gamma_3)), \quad 
    i,j\in\{1,2,3\}.
    \]

    We compute
    \begin{align*}
        \tilde{C}(h_{1},h_{2}h_{3})
        =&
        \phi_\chi
            \bigl(
                \Phi(
                    f_{123}^{*}  f_{1}  f_{23}
                )
            \bigr),
        \\
        \tilde{C}(h_{2},h_{3})
        =&
        \phi_\chi
            \bigl(
                \Phi(
                    f_{23}^{*}  f_{2}  f_{3}
                )
            \bigr),
        \\
        \tilde{C}(h_{1}h_{2},h_{3})
        =&
        \phi_\chi
            \bigl(
                \Phi(
                    f_{123}^{*}  f_{12}  f_{3}
                )
            \bigr)
        ,
        \\
        \tilde{C}(h_{1},h_{2})\ \phi_\chi (f_3^{*}f_3)
        =&
        \alpha_{f_{3}} (\phi_\chi)
            \bigl(
                \Phi(
                    f_{12}^{*}  f_{1}  f_{2}
                )
            \bigr)
        \
        \phi_\chi (f_3^{*}f_3)
        \\
        \overset{(\dagger)}{=}&
        \phi_\chi
            \bigl(
                f_3^{*}
                \Phi(
                    f_{12}^{*}  f_{1}  f_{2}
                )
                f_3
            \bigr)
        \overset{(\ddagger)}{=}
        \phi_\chi
            \bigl(
                \Phi(
                    f_3^{*}f_{12}^{*}  f_{1}  f_{2}f_3
                )
            \bigr),
    \end{align*}
    where we used in $(\dagger)$ the defining property of $\alpha_{f_{3}}$, and in $(\ddagger)$ that $n \Phi(m)n^{*} = \Phi( n \Phi(m) n^{*}) = \Phi(n mn^{*}) $ for functions $n,m
    \in C_c(\gpd, c)
    $ with $n$ supported on a bisection: the first equality holds since the left-hand side is an element of $B$, and the second equality can be checked by hand, using that $\agpd \sset\Iso(\gpd )$  and that $\agpd $ is normal.
    
    We can now rephrase Equation~\eqref{eq:cocycle-claim} to
    \begin{align}\label{eq:cocycle-claim2}
    \phi_\chi
            \bigl(
                \Phi(
                    \underbrace{ f_{123}^{*} f_{1}f_{23}}_{=:F_{1}}
                )
             \bigr)
         \
         \phi_\chi
             \bigl(
                \Phi(
                    \underbrace{f_{23}^{*}  f_{2}f_{3} }_{=:F_{2}}
                )
            \bigr) \
        \overset{!}{\propto} \
        \phi_\chi
            \bigl(
                \Phi(
                    \underbrace{f_{123}^{*}  f_{12}  f_{3}}_{=:F_{3}}
                )
            \bigr)
        \
        \phi_\chi
            \bigl(
                \Phi(
                    \underbrace{f_3^{*}f_{12}^{*}  f_{1}  f_{2}f_3}_{=:F_{4}}
                )
            \bigr).
    \end{align}
    Since $F_i \in C_c(\gpd, c)$ for each $i$, by definition of $\phi_\chi$ and $\Phi$ we have
    
    \begin{align*}
        \phi_\chi\bigl(\Phi(F_{i})\bigr)
        =&
        \sum_{a\in \agpd_{\rho(\chi)}}
            \chi (a)\,
            \Phi(F_{i})(a)
        =
        \sum_{a\in \agpd_{\rho(\chi)}}
            \chi (a)\,
            F_{i}(a).
    \end{align*}
    Note that each $F_{i}$ is the product of functions in $N$, so $F_{i}$ is also supported in a bisection.
    Consequently, 
    there is a unique $x_i \in \agpd_{\rho(\chi)}$ at which $F_i$ does not vanish,
    namely
    \begin{alignat*}{2}
    x_1 &:= \gamma_{123}\inv  \gamma_1 \gamma_{23}, &\qquad
    x_2 &:= \gamma_{23}\inv \gamma_2 \gamma_3, \\
    x_3 &:= \gamma_{123}\inv \gamma_{12}\gamma_3, &\qquad
    x_4 &:= \gamma_3\inv\gamma_{12}\inv\gamma_1 \gamma_2 \gamma_3.
    \end{alignat*}
    We note that each $x_i \in \agpd_{\rho(\chi)}$ because $q(\gamma_{ij}) = q(\gamma_i \gamma_j)$ for $i,j \in \{1,2,3\}$ and $s(x_i) = s(\gamma_3) = \rho(\chi)$.
    
    Equation~\eqref{eq:cocycle-claim2} thus becomes
    \begin{align*}
    \chi(x_{1})\, F_{1}(x_{1}) \ 
    \chi(x_{2})\, F_{2}(x_{2}) \
        \overset{!}{\propto} \
    \chi(x_{3})\, F_{3}(x_{3}) \ 
    \chi(x_{4})\, F_{4}(x_{4}).
    \end{align*}    
    Using the
    fact that each $\chi\in\specB_{u}$ is a projective representation by assumption,
    we compute
    \begin{align*}
    \chi(x_{1})\,
    \chi(x_{2})
    =\chi(x_{1}x_{2})\, c(x_{1},x_{2})
    \text{ and }
    \chi(x_{3})\, \chi(x_{4})
    =
    \chi(x_{3}x_{4})\, c(x_{3},x_{4}).
    \end{align*}
    Note that $x_{1}x_{2}=
    x_{3}x_{4}$, so Equation \eqref{eq:cocycle-claim2} can be rewritten to
    \begin{align}\label{eq:claim3}
    c(x_{1},x_{2})\, F_{1}(x_{1}) \, F_{2}(x_{2}) \
        \overset{!}{\propto} \
    c(x_{3},x_{4})\, F_{3}(x_{3}) \, F_{4}(x_{4}).
    \end{align}    
    Using the formula for the convolution, we compute:
    \begin{align*}
        F_{1} (x_{1})
        =& 
        \uwave{\overline{f_{123} (\gamma_{123})}\,
        \overline{c(\gamma_{123}\inv,\gamma_{123})}}\,
        \uwave{f_{1} (\gamma_{1})} \,
        c(\gamma_{123}\inv, \gamma_{1})\, 
        \dotuline{f_{23}(\gamma_{23})}\,
        c(x_{1}\gamma_{23}\inv,\gamma_{23})
        \\
        F_{2}(x_{2})
        =&
        \dotuline{\overline{f_{23} (\gamma_{23})}}\,
        \overline{c(\gamma_{23}\inv,\gamma_{23})}\,
        \uwave{f_{2} (\gamma_{2})} \,
        c(\gamma_{23}\inv, \gamma_{2})\, 
        \uwave{f_{3}(\gamma_{3})}\,
        c(x_{2}\gamma_{3}\inv,\gamma_{3})
        \\
        F_{3}(x_{3})
        =&
        \uwave{\overline{f_{123} (\gamma_{123})}\,
        \overline{c(\gamma_{123}\inv,\gamma_{123})}}\,
        \dotuline{f_{12} (\gamma_{12})} \,
        c(\gamma_{123}\inv, \gamma_{12})\, \uwave{f_{3}(\gamma_{3})}\,
        c(x_{3}\gamma_{3}\inv,\gamma_{3})
        \\
        F_{4} (x_{4})
        =&
        \dotuline{\overline{f_{3}(\gamma_{3})}}\,
        \dotuline{\overline{f_{12}(\gamma_{12})}}\,
        \uwave{f_{1}(\gamma_{1})}\,
        \uwave{f_{2}(\gamma_{2})}\,
        \dotuline{f_{3}(\gamma_{3})}\,
        \overline{c(\gamma_{3}\inv,\gamma_{3})}\,
        \overline{c(\gamma_{12}\inv,\gamma_{12})}\\
        &
        c(\gamma_{3}\inv, \gamma_{12}\inv)\,
        c(x_{4}\gamma_{3}\inv\gamma_{2}\inv\gamma_{1}\inv,\gamma_{1})\,
        c(x_{4}\gamma_{3}\inv\gamma_{2}\inv,\gamma_{2})\,
        c(x_{4}\gamma_{3}\inv, \gamma_{3})
    \end{align*}
    Note here that the \uwave{wave-underlined factors} cancel on both sides of the alleged Equation~\eqref{eq:claim3} and the \dotuline{dot-underlined factors} amount to a positive factor (which is irrelevant to determine $\propto$).
    Thus, it suffices    to show that
    \begin{align}\begin{split}\label{eq:main-cocycle-claim-for-C}
    &c(x_{1},x_{2})\,
        c(\gamma_{123}\inv, \gamma_{1})\,         c(x_{1}\gamma_{23}\inv,\gamma_{23}) \, 
        \overline{c(\gamma_{23}\inv,\gamma_{23})}\,
        c(\gamma_{23}\inv, \gamma_{2})\,
        c(x_{2}\gamma_{3}\inv,\gamma_{3})
        \\
        \overset{!}{=}\, &
    c(x_{3},x_{4})\,
         c(\gamma_{123}\inv, \gamma_{12})\,
         c(x_{3}\gamma_{3}\inv, \gamma_{3}) \, 
        \overline{c(\gamma_{3}\inv,\gamma_{3})}\,
        \overline{c(\gamma_{12}\inv,\gamma_{12})}\,
        c(\gamma_{3}\inv, \gamma_{12}\inv)\\
        &c(\gamma_{3}\inv\gamma_{12}\inv,\gamma_{1})
        \,c(x_{4}\gamma_{3}\inv\gamma_{2}\inv,\gamma_{2})
        \, c(x_{4}\gamma_{3}\inv, \gamma_{3}).
    \end{split}\end{align}
    The proof of Equation~\eqref{eq:main-cocycle-claim-for-C} is a long exercise in cocycle computations and can be found in the Appendix, Lemma~\ref{lem:main-cocycle-claim-for-C}.
\end{proof}

\begin{lemma}\label{lem:C-cts}
     If $\qsec $ is a continuous section, then the 2-cocycle $C^{\qsec }$ from Proposition~\ref{prop:def:cocyle-for-Weyl} is continuous.
\end{lemma}

\begin{proof}
    Let $C:= C^{\qsec }.$
    Suppose we are given a convergent net in $ (\qgpd  \ltimes \specB  )\2$, say
    \[
        (x_{i}, y_{i})
        :=
        \Bigl(
            \bigl(\dot{\gamma}_{i} , \tilde{\alpha}_{\dot{ \tau}_{i}} (\chi_{i})\bigr),
            (\dot{ \tau}_{i} , \chi_{i})
        \Bigr)
        \overset{i}{\longrightarrow}
        \Bigl(
            \bigl(\dot{\gamma} , \tilde{\alpha}_{\dot{ \tau}} (\chi)\bigr),
            (\dot{ \tau} , \chi )
        \Bigr)
        :=
        (x, y),
        \,i\in\Lambda.
    \]
    
    The assumption that $\lim_{i} (x_{i},y_{i})= (x,y)$ means that
    \[
        \lim_{i} \dot{\gamma}_{i} = \dot{\gamma}
        \text{ and }
        \lim_{i} \dot{ \tau}_{i} = \dot{ \tau}
        \text{ in }
        \qgpd , \text{ and }
        \lim_{i} \chi_{i} = \chi
        \text{ in } \specB .
    \]
    Recall that the latter implies that $u_{i} := \rho(\chi_{i})$ converges to  $\rho(\chi)=:u$.
    
    Let $ \gamma := \qsec (\dot{\gamma}), \tau:= \qsec (\dot{\tau})$ be the representative of $\dot{\gamma}$, resp.\
    $\dot{ \tau},$ in $\supp' (f_{\dot{\gamma}}),$ resp.\
    $\supp' (f_{\dot{ \tau}})$. Note that there exists a (unique) $a\in\agpd_{r(\tau)}$ with $\gamma a \tau = \qsec (q(\gamma\tau))\in \supp'(f_{q(\gamma\tau)}).$
    
    Similarly, for each $i$, let $\gamma_{i}:=\qsec (\dot{\gamma}_{i}) ,  \tau_{i}:=\qsec (\dot{\tau}_{i})$ be the representative of $\dot{\gamma}_{i},$ resp.\
    $\dot{ \tau}_{i},$ in $\supp' (f_{\dot{\gamma}_{i}}),$ resp.\
    $\supp' (f_{\dot{ \tau}_{i}})$. Again, for each $i$, there exists a (unique) $a_{i} \in\agpd_{r(\tau_{i})}$ such that $\gamma_{i} a_{i} \tau_{i}= \qsec (q(\gamma_{i}\tau_{i}))\in \supp'(f_{q(\gamma_{i}\tau_{i})}).$

    Since $\dot{\gamma}_{i}\overset{i}{\to}\dot{\gamma}$, continuity of $\qsec $ implies $\gamma_{i} = \qsec (\dot{\gamma}_i)\overset{i}{\to}\qsec (\dot{\gamma})=\gamma.$ Similarly, $\dot{ \tau}_{i} \overset{i}{\to} \dot{ \tau}$ implies $\tau_i\overset{i}{\to}\tau$. Since $\supp'(f_{\dot{\gamma}})$ and $\supp'(f_{\dot{\tau}})$ are open neighborhoods of $\gamma,$ resp.\ $\tau$, there thus exists $i_{1}$ so that for all $i\geq i_{1}$, we have $\gamma_{i} \in \supp'(f_{\dot{\gamma}})$ and $\tau_{i}\in \supp'(f_{\dot{ \tau}})$.
    Since $q(\gamma_{i} \tau_{i})=\dot{\gamma}_{i}\dot{ \tau}_{i} \overset{i}{\to} \dot{\gamma}\dot{ \tau}=q(\gamma\tau)$ and $\qsec $ is continuous, we also have 
    \[
        \gamma_i a_i \tau_i = \qsec (q(\gamma_{i} \tau_{i})) \overset{i}{\to} \qsec (q(\gamma\tau)) = \gamma a \tau,
    \]
    so there exists $i_{2}$
    so that for all $i\geq i_{2}$, we have
    $ \gamma_{i} a_{i} \tau_{i} \in \supp'(f_{q(\gamma\tau)})$.
    
    In order to show that $\lim_{i} C(x_{i},y_{i})=C(x,y)$, consider $F_{i} := f_{q(\gamma_{i}   \tau_{i})}^{*} f_{\dot{\gamma}_{i}}f_{\dot{ \tau}_{i}}$ 
    and recall that $C(x_i, y_i) = \frac{\phi_{\chi_i}(\Phi(F_i))}{|\phi_{\chi_i}(\Phi(F_i))|}$.
    The fact that $F_i$ is a product of elements of $N$ and hence is supported in a bisection means that $\agpd_{u_{i}} \cap\supp'(F_{i})$ consists of the single point
    \[
        \tau_{i}\inv a_{i}\inv \tau_{i} =
        (\gamma_{i} a_{i} \tau_{i} )\inv\, \gamma_{i} \,  \tau_{i}
        \in
        \supp' (f_{q(\gamma_{i}   \tau_{i})})\inv
        \cdot
        \supp'(f_{\dot{\gamma}_{i}})
        \cdot
        \supp' (f_{\dot{ \tau}_{i}})
        ,
    \]
    so that
    \begin{align*}
        \phi_{\chi_{i}} \left( \Phi(F_{i})\right)
        =
        \chi_{i} ( \tau_{i}\inv a_{i}\inv \tau_{i} )\,F_{i} ( \tau_{i}\inv a_{i}\inv \tau_{i})  .
    \end{align*}
    We compute (cf.~the proof of Proposition~\ref{prop:def:cocyle-for-Weyl}) that 
    \begin{align*}
        F_{i}(\tau_{i}\inv a_{i}\inv \tau_{i})
        =&
        \overline{f_{q(\gamma_{i}   \tau_{i})} (\gamma_{i} a_{i} \tau_{i})}\,
        \overline{c((\gamma_{i} a_{i} \tau_{i})\inv,\gamma_{i} a_{i} \tau_{i})}\\
        &f_{\dot{\gamma}_{i}} (\gamma_{i})\,
        c((\gamma_{i} a_{i} \tau_{i})\inv, \gamma_{i})\, 
        f_{\dot{\tau}_{i}}(\tau_{i})\,
        c(\tau_{i}\inv a_{i}\inv,\tau_{i}).
    \end{align*}
    In particular, since we assumed that $f_{\dot{\kappa}}(\qsec (\dot{\kappa}))> 0$ for any $\dot{\kappa}\in\qgpd $,
    we have
    \begin{align*}
        C(x_{i}, y_{i})
        =&
        \overline{c((\gamma_{i} a_{i} \tau_{i})\inv,\gamma_{i} a_{i} \tau_{i})}\,
        c((\gamma_{i} a_{i} \tau_{i})\inv, \gamma_{i})\, 
        c(\tau_{i}\inv a_{i}\inv,\tau_{i}) \,
        \chi_{i} ( \tau_{i}\inv a_{i}\inv \tau_{i} ) 
        .
    \end{align*}
    As $\gamma_{i}\to\gamma$, $\tau_{i}\to \tau$, and $\gamma_{i} a_{i} \tau_{i}  \to  \gamma a \tau$ (all by continuity of $\qsec $), we have that $a_{i}\to a$. Since $\chi_{i}\to \chi$, it follows that
    \[
        \chi_{i} ( \tau_{i}\inv a_{i}\inv \tau_{i} )
        \to
        \chi ( \tau\inv a\inv \tau).
    \]
    Since $c$ is continuous, we conclude that
    \[
        C(x_{i},y_{i})
        \to
        \overline{c((\gamma a \tau)\inv,\gamma a \tau)}\,
        c((\gamma a \tau)\inv, \gamma)\, 
        c(\tau\inv a\inv,\tau) \,
        \chi ( \tau\inv a\inv \tau).
    \]
    Note that the right-hand side is, by the same argument as for $C(x_{i},y_{i})$, exactly $C(x,y).$ This concludes our proof.
\end{proof}

\pagebreak[3]\section{Computing the Weyl twist}
 Our standing assumptions   can be found on page~\pageref{assumption:Sigma}. 
 In the main result of this section (Theorem~\ref{thm:psi}), we prove that every section $\qsec $ of the quotient map $q\colon  \gpd \to \agpd$ induces a $\mbb{T}$-equivariant isomorphism of the (discrete) groupoids 
$(\qgpd  \ltimes \specB ) \times_{C^\qsec } \mbb{T}$ and $\Sigma_{(A,B)}$. Moreover, if $\qsec $ is continuous, then the isomorphism is a homeomorphism, meaning we have explicitly identified not only the Weyl groupoid but also the Weyl twist of the Cartan pair $(A,B)$, see Corollary~\ref{cor:cartan}.

We begin by recalling the classical construction of a twist $G \times_\sigma \mathbb T$ from a 2-cocycle $\sigma$ on a groupoid $G$ (cf.~\cite[Chapter I, Proposition 1.14]{renault} or \cite{DGNRW:Cartan}). As a set, we have $G \times_\sigma \mathbb T = G \times \mathbb  T$; we set $(G\times_\sigma \mathbb T)\2 = \{ (g, \lambda), (h, \mu): (g,h) \in G\2\}$.  The multiplication is given by $(g, \lambda) (h, \mu) = (gh, \sigma(g,h) \lambda \mu).$ When $\sigma$ is continuous, the product topology on $G\times_\sigma \mathbb{T}$ makes it a topological groupoid. Note that there is a continuous action of $\mathbb  T$ on $G \times_\sigma \mathbb T$, given by $\lambda \cdot (g, \mu) = (g, \lambda \mu)$, and $(G \times_\sigma \mathbb T) / \mathbb T \cong G.$ Similarly (cf.~\cite[Section 3]{Kumjian:diag}), the Weyl twist $\Sigma_{(A, B)}$ always admits  an action of $\mathbb T$, given by $\lambda \cdot \llbracket \alpha_n(\phi_\chi), n, \phi_\chi\rrbracket = \llbracket \alpha_n(\phi_\chi), \lambda n, \phi_\chi\rrbracket,$ so that $\Sigma_{(A, B)}/\mathbb T = \gpd_{(A,B)}.$  

\begin{thm}\label{thm:psi}
    As in Proposition~\ref{prop:def:cocyle-for-Weyl}, let $\qsec \colon\qgpd \to\gpd $ be a section for  $q\colon\gpd \to\qgpd $ and for each $\dot{\gamma}\in \qgpd $, choose one $f_{\dot{\gamma}}\in N$ such that $f_{\dot{\gamma}}(\qsec (\dot{\gamma})) >0$.
    Define
    \[
        \psi_{\qsec }\colon (\qgpd  \ltimes \specB ) \times_{C^\qsec } \mbb{T} \to \Sigma_{(A,B)},
        \quad
        \psi_{\qsec} (\dot{\gamma},\chi, \lambda) = \llbracket \phi(\tilde{\alpha}_{\dot{\gamma}}(\chi)), \lambda f_{\dot{\gamma}}, \phi_{\chi}\rrbracket.
    \]
     Then the map $\psi_{\qsec }$ is a $\mbb{T}$-equivariant isomorphism of  \textup(discrete\textup) groupoids 
     which does not depend on the choice of $f_{\dot{\gamma}}$. 

    Moreover, if $\qsec $ is a continuous section, then $\psi_{\qsec }$ is a 
    homeomorphism.
\end{thm}

\begin{cor}
\label{cor:cartan}
    If $\qsec $ is continuous, then the Cartan pair $(C_r^{*}(\gpd ,c), C_r^{*}(\agpd ,c))$ is isomorphic to the pair $\bigl(C_r^{*}((\qgpd \ltimes\specB )\times_{C^\qsec }\mbb{T}),C_0(
    \specB )\bigr)
    $.
\end{cor}

\begin{proof}
Since the Weyl groupoid is isomorphic, as a topological groupoid, to $\qgpd \ltimes\specB $ by Theorem~\ref{thm:varphi}, and the Weyl twist is also topologically isomorphic to $(\qgpd \ltimes\specB )\times_{C^\qsec }\mbb{T}$ by Theorem~\ref{thm:psi},  \cite[Theorem 5.9]{renault-cartan} implies that $C_r^{*}(\gpd ,c)$ is isomorphic to $C_r^{*}(\qgpd \ltimes \specB ,(\qgpd \ltimes \specB )\times_{C^\qsec } \mbb{T})$, and that isomorphism carries $C_r^{*}(\agpd ,c)$ onto $C_0((\qgpd  \ltimes \specB )^{(0)}) \cong C_0(\specB)    
$. 
\end{proof}

The remainder of this section is devoted to proving Theorem~\ref{thm:psi}. In order to do so, we will need the following analogue of \cite[Lemma 5.4]{DGNRW:Cartan} that
states that every element of the Weyl twist can be represented by a function supported in a bisection and scaled by an explicitly computed factor.

\begin{prop}
\label{prop:Weyl-twist-elt-description}
    Let $\llbracket \alpha_{n}(\phi_\chi), n, \phi_\chi\rrbracket$ be an  arbitrary element of the Weyl twist $\Sigma_{(A,B)}$, and let $\dot{\gamma}\in\qgpd $ be such that
    $
        \varphi(\dot{\gamma},\chi)
        =
        [\alpha_{n}( \phi_\chi ), n,  \phi_\chi ].
    $
    If $f\in N$ satisfies $\dot{\gamma} \in q(\supp'(f)),$ then
    \[ 
        \llbracket \alpha_{n}(\phi_\chi), n, \phi_\chi\rrbracket
        =
        \llbracket \alpha_{f}(\phi_\chi), \lambda f, \phi_\chi\rrbracket
    \quad\text{where}\quad
        \lambda
        =
        \frac{\phi_{\chi} \bigl( \Phi (f^{*} \, n) \bigr)}{\abs{\phi_{\chi} \bigl( \Phi (f^{*} \, n) \bigr)}}
        \in\mbb{T}.
    \]
    \end{prop}

Note that by the surjectivity of $\varphi$ (see Theorem~\ref{thm:varphi}), such an element  $\dot{\gamma}$ always exists. Moreover, Proposition~\ref{prop:from-alpha-to-tilde-alpha}\ref{item:from-alpha-to-tilde-alpha} shows that $\alpha_{f}(\phi_{\chi}) = \phi_{\tilde{\alpha}_{\dot{\gamma}}(\chi)}$ for any $f$ as specified in the proposition.

\begin{proof}
    Write $x:=\phi_\chi\in\dom(n)$. 
    We have to find two elements $b,b'\in B$ such that $x (b), x (b') >0$ and $nb = (\lambda f)b'$.
    By assumption on $f$ and definition of $\varphi$, we have
   \begin{equation*}
        [\alpha_{f}( x ), f,  x ]
        =
        \varphi(\dot{\gamma},\chi)
        =
        [\alpha_{n}( x ), n,  x ].
    \end{equation*}
    Proposition \ref{prop:Weyl-gpd-similar-to-twist} therefore tells us that there exist $b_1, b_2 \in B$ such that $f b_1 = n b_2$ and $x(b_1), x(b_2) \not= 0$. In fact, the construction of $b_1, b_2$ in the proof of that proposition gives $x(b_2) > 0$.  Thus, if we set 
    \[b' = \frac{|x(b_1)|}{x(b_1)} b_1,\qquad b = b_2, \qquad  \lambda = \frac{x(b_1)}{|x(b_1)| },\]
    then $x(b), x(b')>0$ and $\lambda f b' = f b_1 = n b_2$ as desired.
    
    To see that $\lambda$ can be equivalently written as in the statement of the proposition, recall from Equation~\eqref{eq:def:b_1} that 
\[ b_1 = \Omega(k) f^{*} n \,\Omega(\overline{k}),\]
where $k \in C_0(\widehat B)$ is supported on $\dom(f) \cap \dom(n)$ and satisfies $k(x) = 1$.  As the conditional expectation $\Phi: A \to B$ is $B$-linear, it follows that 
\[ x(b_1) = \phi_\chi(\Phi(b_1))= \phi_\chi( \Omega(k) \Phi( f^{*} n ) \Omega(\overline{k}))= \phi_\chi(\Phi(f^{*} n)).\]
This yields the asserted 
expression for $\lambda$.
\end{proof}
As we will frequently need to explicitly compute the constant $\lambda$ appearing in Proposition~\ref{prop:Weyl-twist-elt-description}, the following corollary will be helpful.
\begin{cor}\label{cor:explicit-lambda}
   In the setting of Proposition~\ref{prop:Weyl-twist-elt-description}, suppose further that $n \in N$. If the unique element $\gamma$ of $\supp'(n)\cap \agpd_{\rho(\chi)}$ is also an element of  $\supp'(f)$, then
   \[ 
        \llbracket \alpha_{n}(\phi_\chi), n, \phi_\chi\rrbracket
        =
        \llbracket \alpha_{f}(\phi_\chi), \lambda f, \phi_\chi\rrbracket
    \quad\text{where}\quad
        \lambda
        = \frac{\overline{f(\gamma )} }{\abs{f(\gamma )}}\frac{ n(\gamma)}{\abs{ n(\gamma)}}.
    \]
\end{cor}

\begin{proof}
Since $f^{*} n \in N$, \cite[Proposition 3.13]{DGNRW:Cartan} tells us that $\Phi(f^{*}n)$ is  given by restriction to $\agpd$.  In particular, $\Phi(f^{*} n) \in C_c(\agpd, c)$. The definition of $\phi_\chi$  therefore yields  
\[\phi_\chi( \Phi(f^{*} n) )= \sum_{a\in \agpd_{\rho(\chi)}} \chi(a
) f^{*} n (a) = \sum_{a\in \agpd_{\rho(\chi)}} \chi(a) \sum_{r(\eta) = \rho(\chi)} f^{*}(a\eta) n(\eta\inv ) c(a \eta, \eta\inv ).\]
Since $f$ and $n$ are supported in bisections, 
and both have the same element of $\gpd_{\rho(\chi)}$ in their support,
only the summand
where
$\eta:=\gamma\inv$ and $a:=\rho(\chi)$ 
does
not vanish, i.e.\ the sum above simplifies to a single term:
\[ \phi_\chi(\Phi(f^*n)) = \chi(\rho(\chi)) \overline{f(\gamma \rho(\chi)\inv ) c(\gamma \rho(\chi)\inv , \rho(\chi)\gamma\inv )} n(\gamma) c(\rho(\chi)\gamma\inv , \gamma).\]
As $\chi(\rho(\chi)) = 1$ and $c(\gamma^{-1}, \gamma) = c(\gamma, \gamma^{-1}),$ we conclude that
$
    \phi_\chi( \Phi(f^{*} n)) = \overline{f(\gamma)} n(\gamma)$ and hence the desired formula for $\lambda$.
\end{proof}

The proof of Theorem~\ref{thm:psi}  now proceeds through a series of lemmata. To simplify notation, we will write $\psi$ for $\psi_{\qsec }$ and~$C$ for~$C^\qsec $ in the proofs, tacitly using the chosen~$\qsec $.

\begin{lemma}    
\label{lem:psi-surjective}
    The map $\psi_{\qsec }$ is surjective and does not depend on the choice of $f_{\dot{\gamma}}$.
\end{lemma}

\begin{proof}
    Let $\psi:= \psi_{\qsec }.$
    Since  
    $
        \rho(\chi)=s(\qsec (\dot{\gamma}))\in s(\supp'(f_{\dot{\gamma}}))
    $, Proposition~\ref{prop:from-alpha-to-tilde-alpha}\ref{item:dom-f} implies that $\phi_\chi\in\dom(f_{\dot{\gamma}}).$ Therefore $\psi_\qsec(\dot{\gamma}, \chi, \lambda)$ is a well-defined element of $\Sigma_{(A,B)}$.
    
    Let $\llbracket \alpha_{n}(x), n, x\rrbracket$ be an arbitrary element of $\Sigma_{(A,B)}$. By the surjectivity of $\varphi$ (see Proposition~\ref{thm:varphi}) there indeed exists $\dot{\gamma}$ such that $\varphi(\dot{\gamma},\chi)=[ \alpha_{n}(x), n, x ],$ where $x=\phi_{\chi}$ (see Lemma~\ref{lem:phi-surjective}). It then follows from Proposition~\ref{prop:Weyl-twist-elt-description} that, since $f_{\dot{\gamma}}\in N$ has $\dot{\gamma}\in q(\supp'(f))$, we have
    \[ 
        \llbracket \alpha_{n}(x), n, x\rrbracket
        =
        \psi(\dot{\gamma}, \chi, \lambda)
        \quad\text{where}\quad
        \lambda
        =
        \frac{\phi_{\chi} \bigl( \Phi (f_{\dot{\gamma}}^{*} \, n) \bigr)}{\abs{\phi_{\chi} \bigl( \Phi (f_{\dot{\gamma}}^{*} \, n) \bigr)}}
        \in\mbb{T},
    \] 
    which proves surjectivity of $\psi$. 
    
    Now suppose $f$ is another element of $N$ with $f(\qsec (\dot{\gamma}))>0$. Since $\gamma:=\qsec (\dot{\gamma})$ is in both $\supp'(f)$ and  $\supp'(f_{\dot{\gamma}})$,  Corollary~\ref{cor:explicit-lambda} yields  $ 
    \llbracket \alpha_{f_{\dot{\gamma}}}(\phi_\chi), \lambda f_{\dot{\gamma}}, \phi_\chi\rrbracket=\llbracket \alpha_{f}(\phi_\chi), \mu f, \phi_\chi\rrbracket$ for
    \[
        \mu =
        \frac{\overline{f(\gamma )} }{\abs{f(\gamma )}}\frac{ \lambda f_{\dot{\gamma}}(\gamma)}{\abs{ f_{\dot{\gamma}}(\gamma)}}
    \]
    Since we have both $f_{\dot{\gamma}}(\gamma) > 0$ and $f(\gamma)  > 0$ by hypothesis, we conclude $\mu=\lambda$. Thus, the map $\psi_\qsec$ does not depend on the choice of $f_{\dot{\gamma}}$.
\end{proof}

\begin{lemma}
\label{lem:top-prin=>psi-injective}
    The map $\psi_{\qsec }$ is injective.
\end{lemma}

\begin{proof}
   Let $\psi:= \psi_{\qsec }.$
If $\psi (\dot{\gamma}, \chi, \lambda) = \psi (\dot{ \tau}, \chi', \lambda')$, then
   \begin{equation}
   \label{eq:useful}
    \llbracket \phi(\tilde{\alpha}_{\dot{\gamma}}(\chi)), \lambda f_{\dot{\gamma}}, \phi_{\chi}\rrbracket
    =
    \llbracket \phi(\tilde{\alpha}_{\dot{ \tau}}(\chi')), \lambda' f_{\dot{ \tau}}, \phi_{\chi'}\rrbracket .
   \end{equation}
   Comparing the right-most components yields $\chi=\chi'$ by Lemma \ref{lem:phionetoone}. 
    In particular, $s(\dot{\gamma})=\rho(\chi)=s(\dot{\tau})$. Since $\dot{\gamma}\in q(\supp'(f_{\dot{\gamma}}))$ and $\dot{\tau}\in q(\supp'(f_{\dot{\tau}}))$, 
    the second assertion of
    Lemma~\ref{lem:3-implies-2}\ref{item:Weyl-twist-implies-points}
    then yields $\dot{\gamma}=\dot{ \tau}.$  Equation \eqref{eq:useful} then implies the existence of $b,b'\in B$ such that $\phi_\chi (b),\phi_\chi (b') >0$ and \[(\lambda f_{\dot{\gamma}})b = (\lambda' f_{\dot{ \tau}})b'=(\lambda' f_{\dot{\gamma}})b'.\]
   If we multiply both sides by $f_{\dot{\gamma}}^{*}$ and evaluate at $\phi_\chi$, we get
   \[
        \lambda \,\phi_\chi(f_{\dot{\gamma}}^{*}f_{\dot{\gamma}})
        \,\phi_\chi(b)
        =
        \lambda'\,
        \phi_\chi(f_{\dot{\gamma}}^{*}\,    f_{\dot{\gamma}}) \phi_\chi(b').
   \]
  Since $\rho(\chi)=s(\dot{\gamma})
  =s(\qsec({\dot{\gamma}}))
  $, 
   Proposition~\ref{prop:from-alpha-to-tilde-alpha}\ref{item:dom-f}
   implies that $ \phi_\chi(f_{\dot{\gamma}}^{*}    f_{\dot{\gamma}}) > 0$.
   As $\phi_\chi (b)$ and $\phi_\chi (b')$ are positive also and $\lambda, \lambda'\in\mbb{T}$, we conclude that  $\lambda=\lambda'$.
\end{proof}

\begin{lemma}\label{lem:psi-T-equiv}
    The map $\psi_{\qsec}$ is a $\mbb{T}$-equivariant groupoid homomorphism.
\end{lemma}
\begin{proof}
    Let $\psi:= \psi_{\qsec }$ and $C:= C^{\qsec }.$    Regarding $\mbb{T}$-invariance, recall that for any $z\in\mbb{T}$, 
    \begin{align*}
        z\, \psi(\dot{\gamma},\chi,\lambda)
        =&
        z\, \llbracket \phi(\tilde{\alpha}_{\dot{\gamma}}(\chi)), \lambda f_{\dot{\gamma}}, \phi_{\chi}\rrbracket
        =
        \llbracket \phi(\tilde{\alpha}_{\dot{\gamma}}(\chi)), z \lambda f_{\dot{\gamma}}, \phi_{\chi}\rrbracket
        \\
        =&
        \psi(\dot{\gamma},\chi,z \lambda)
        =
        \psi\bigl(z\cdot (\dot{\gamma},\chi,\lambda)\bigr).
    \end{align*}
    
    To see that $\psi$ is multiplicative, recall that each composable pair in $ (\qgpd  \ltimes \specB  ) \times_{C}\mbb{T}$ is of the form
       \[
            \Bigl(
                (\dot{\gamma}_{1},\tilde{\alpha}_{\dot{\gamma}_{2}}(\chi),\lambda_{1})
                ,
                (\dot{\gamma}_{2},\chi,\lambda_{2})
            \Bigr),
            \text{ where }
            \chi\in \specB_{s(\dot{\gamma_{2}})} 
            \text{ and }
            s(\dot{\gamma_{1}})=r(\dot{\gamma_{2}}).
       \]
    If we let $h_{1}:= (\dot{\gamma}_{1},\tilde{\alpha}_{\dot{\gamma}_{2}}(\chi))$ and $h_{2}:= (\dot{\gamma}_{2},\chi)\in \qgpd  \ltimes \specB  $, then the composition of the above composable pair in $ (\qgpd  \ltimes \specB  ) \times_{C} \mbb{T}$ is given by $( q(\gamma_{1} \gamma_{2}) ,\chi,C(h_{1}, h_{2}) \lambda_{1} \lambda_{2}).$ 
    Then
       \begin{align*}
            \psi ( (h_{1},\lambda_{1})\cdot ( h_{2},\lambda_{2}))
            =&\llbracket \phi(\tilde{\alpha}_{ q(\gamma_{1} \gamma_{2}) }(\chi)), C(h_{1},h_{2})\lambda_{1} \lambda_{2} \, f_{q(\gamma_{1}\gamma_{2})}, \phi(\chi)\rrbracket
            .
       \end{align*}
       On the other hand,
       \begin{align*}
            \psi (h_{1},\lambda_{1})
            \,
            \psi(h_{2},\lambda_{2})
            =&
            \llbracket \phi( \tilde{\alpha}_{\dot{\gamma}_{1}}(\tilde{\alpha}_{\dot{\gamma}_{2}}(\chi))), \lambda_{1} f_{\dot{\gamma}_{1}}, \phi(\tilde{\alpha}_{\dot{\gamma}_{2}}(\chi))\rrbracket
            \,
            \llbracket \phi(\tilde{\alpha}_{\dot{\gamma}_{2}}(\chi)), \lambda_{2} f_{\dot{\gamma}_{2}}, \phi(\chi)\rrbracket
            \\
            =&
            \llbracket \phi( \tilde{\alpha}_{ q(\gamma_{1} \gamma_{2}) }(\chi)), \lambda_{1} \lambda_{2}\, (f_{\dot{\gamma}_{1}}
            f_{\dot{\gamma}_{2}}), \phi(\chi)\rrbracket ,
        \end{align*}
        where the last equality follows from the definition of the groupoid structure on $\Sigma_{(A,B)}$ and because $\tilde{\alpha}$ is a homomorphism.

        Let $n:= \lambda_{1} \lambda_{2} (f_{\dot{\gamma}_{1}}
        f_{\dot{\gamma}_{2}}).$ Note that $n$ is an element of $N$ with $\supp'(n) = \supp' (f_{\dot{\gamma}_1}) \cdot \supp' (f_{\dot{\gamma}_2})$; in particular,
       \[
            \dot{\gamma}_1 \dot{\gamma}_2 = q(\gamma_{1}\gamma_{2}) \in q\left( \supp' (f_{\dot{\gamma}_1}) \cdot \supp' (f_{\dot{\gamma}_2}) \right) =q(\supp'(n)).
        \]
        Thus,
        \[
            [
                \phi( \tilde{\alpha}_{ q(\gamma_{1} \gamma_{2}) }(\chi)), n, \phi(\chi)
            ]
            =
            \varphi (\dot{\gamma}_1 \dot{\gamma}_2, \chi),
        \]
        and so if we apply Proposition~\ref{prop:Weyl-twist-elt-description}, we get 
        \[
            \llbracket \phi( \tilde{\alpha}_{ q(\gamma_{1} \gamma_{2}) }(\chi)), n, \phi(\chi)\rrbracket
            =
            \llbracket \phi( \tilde{\alpha}_{ q(\gamma_{1} \gamma_{2}) }(\chi)), \lambda f, \phi(\chi)\rrbracket
        \]
        where 
        \[
            f:= f_{q(\gamma_{1}\gamma_{2})} 
            \quad\text{and}\quad
            \lambda :=
            \frac{\phi_{\chi}\bigl( \Phi (f^{*} n)\bigr) }{\abs{\phi_{\chi}\bigl( \Phi (f^{*} n)\bigr)}}.
        \]
        By definition of $C$
        (see Proposition~\ref{prop:def:cocyle-for-Weyl})
        and $n$, we have
        \[
            \lambda_{1}\lambda_{2} C(h_{1},h_{2})
            =
            \lambda_{1}\lambda_{2} \frac{\phi_{\chi}\bigl( \Phi (f^{*} f_{\dot{\gamma}_{1}}f_{\dot{\gamma}_{2}})\bigr) }{\abs{\phi_{\chi}\bigl( \Phi (f^{*} f_{\dot{\gamma}_{1}}f_{\dot{\gamma}_{2}})\bigr)}}
            =
            \lambda.
        \]
        We conclude that $\psi (h_{1},\lambda_{1})
            \,
            \psi(h_{2},\lambda_{2})
            =
            \psi ( (h_{1},\lambda_{1})\cdot (h_{2},\lambda_{2})),$
        as claimed.
\end{proof}

\begin{lemma}%
\label{lem:csec-cts=>psi-cts}
    If $\qsec $ is a continuous section, the map $\psi_{\qsec }$ is continuous.
\end{lemma}

\begin{proof}
  Let $\psi:= \psi_{\qsec } $ and $C := C^\qsec$.
 Assume $x_{i}:=(\dot{\gamma}_{i}, \chi_{i}, \mu_{i})$ is a net in $ ( \qgpd  \ltimes \specB  ) \times_{C} \mbb{T}$ that converges to $x:=(\dot{\gamma}, \chi, \mu)$. If $f:= f_{\dot{\gamma}}$, then a basic open set around $\psi(x)$ (according to \cite[Lemma~4.16]{renault-cartan}) is of the form
    \begin{equation}\label{eq:def:mfU}
        \mf{U}(U,V,f)
        :=
        \Bigl\{
            \llbracket
                \alpha_{f} (\phi_\nu),
                \kappa f, 
                \phi_\nu
            \rrbracket
            \,\vert\,
            \phi_{\nu}\in U, \kappa\in V
        \Bigr\},
    \end{equation}
    where $U\sset\widehat{B}$ is an open neighborhood around $\phi_{\chi}$ and $V\sset\mbb{T}$ an open neighborhood around $\mu$. We need to show that $\psi(x_{i})\in \mf{U}$ for large enough $i$. By definition of $\psi$, we have
    \[
        \psi(x_{i})
        =
        \llbracket
            \alpha_{f_{i}} (\phi_{\chi_{i}}),
            \mu_{i} f_{i}, 
            \phi_{\chi_{i}}
        \rrbracket,
        \quad
        \text{where }f_{i}:=f_{\dot{\gamma}_{i}}.
    \]
     Let $\gamma_{i}:= \qsec (\dot{\gamma}_{i})$ and $\gamma := \qsec (\dot{\gamma})$ be the representatives in the open supports of $f_i$ and $f$ respectively. Note that, since $\supp'(f)$ is an open neighborhood of $\gamma$, it follows from $\dot{\gamma}_{i}\to\dot{\gamma}$ and continuity of $\qsec $ that $\gamma_{i} \in \supp'(f)$ for all $i$ larger than some $i_{1}$. 
    Since we also have $f_i(\gamma_i) > 0$,  Corollary~\ref{cor:explicit-lambda} implies that for     $i\geq i_{1}$,
    \[
        \psi(x_{i})
        =
        \llbracket
            \alpha_f(\phi_{\chi_{i}}),
            \lambda_{i} f, 
            \phi_{\chi_{i}}
        \rrbracket,
        \quad\text{where}\
        \lambda_{i}
        :=
        \frac{\overline{f(\gamma_{i} )} }{\abs{f(\gamma_{i} )}}\frac{ \mu_i f_i(\gamma_{i})}{\abs{ f_i(\gamma_{i})}} = \mu_i \frac{\overline{f(\gamma_{i} )} }{\abs{f(\gamma_{i} )}}.
    \]
    To show that $\psi(x_{i})$ is an element of $\mf{U}(U,V,f)$ for all large enough $i$, we must show that $\lambda_{i}\in V$ and $\phi_{\chi_{i}}\in U$. For the latter, note that as $\chi_{i}\to \chi$ and as $\nu\mapsto\phi_\nu$ is continuous, we know that $\phi_{\chi_{i}}\in U$ for all $i$ larger than some $i_2$. For the former, note that, since $\gamma_{i}\to\gamma$ and $f$ is a continuous function with $f(\gamma)=\overline{f(\gamma)}>0$, we have that $\frac{\overline{f(\gamma_{i})} }
    {\abs{f(\gamma_{i})} }= \lambda_{i}\overline{\mu_{i}}$ converges to $1$. Since $\mu_i\to\mu$ by assumption, this implies that $\lambda_{i}$ is in the neighborhood $V$ of $\mu$ for all large enough $i$.
\end{proof}

\begin{lemma}
    If $\qsec $ is a continuous section, then $\psi_{\qsec }$ is open.
\end{lemma}

\begin{proof}
    Let $\psi:= \psi_{\qsec }.$
    Since $\psi$ is bijective by Lemmata~\ref{lem:psi-surjective} and~\ref{lem:top-prin=>psi-injective}, we may show that $\psi\inv$ is continuous. Assume that we have a convergent net
    $(Y_i)_{i\in \Lambda}$
    in $\Sigma_{(A,B)}$. 
    By Proposition~\ref{prop:Weyl-twist-elt-description}, we may write
    \[
        Y_{i}
        :=
        \llbracket
            \alpha_{f_{\dot{\gamma}_{i}}} (\phi_{\chi_{i}}),
            \lambda_{i} f_{\dot{\gamma}_{i}},
            \phi_{\chi_{i}}
        \rrbracket
        \to
        \llbracket
            \alpha_{f_{\dot{\gamma}}} (\phi_{\chi}),
            \lambda f_{\dot{\gamma}},
            \phi_{\chi}
        \rrbracket
        ,
        \quad i\in\Lambda,
    \]
    where 
    $f_{\dot \gamma},f_{\dot \gamma_i}\in N$.
Let us write $f:= f_{\dot{\gamma}}$ and $f_{i}:= f_{\dot{\gamma}_{i}}$. We have to show that $(\dot{\gamma}_{i}, \chi_{i}, \lambda_{i})$ converges to $(\dot{\gamma}, \chi, \lambda)$ in $ ( \qgpd  \ltimes \specB  ) \times\mbb{T}$.

    Convergence in $\Sigma_{(A,B)}$ implies that for any open neighborhood $U \subseteq \dom(f) \subseteq \widehat{B}$ of $\phi_\chi$ and for any open neighborhood $V \subseteq \mbb{T}$ of $\lambda$, there exists $j_{U,V}$ such that $Y_i \in \mf{U}(U,V,f)$ for $i \geq j_{U,V}$ (see Equation~\eqref{eq:def:mfU} for the definition of $\mf{U}$). In particular, for each $i \geq j_{U,V}$, we have $\phi_{\chi_i}\in U$ and
    \begin{equation}\label{eq:Yi-new}
        Y_{i}
        =
        \llbracket
            \alpha_{f} (\phi_{\chi_i}),
            \kappa_{i} f,
            \phi_{\chi_i}
        \rrbracket
    \end{equation}
    for some $\kappa_i\in V$. Thus $\phi_{\chi_i} \to \phi_\chi$, so $\chi_i \to \chi$ by Lemma \ref{lem:phi-open}. Because of Equation~\eqref{eq:Yi-new}, we can use Lemma~\ref{lem:3-implies-2}\ref{item:Weyl-twist-implies-points} to conclude that, if $i \geq j_{U,V}$,
    \begin{equation}\label{eq:final-ineq}
        \overline{\kappa_{i}}\lambda_{i}\,
        \phi_{\chi_{i}} \bigl(\Phi(f^{*}f_{i})\bigr)
        =
        \phi_{\chi_{i}} \bigl(\Phi((\kappa_{i}f)^{*}\,(\lambda_{i}f_{i}))\bigr)
        >0.
    \end{equation}

  By Proposition~\ref{prop:from-alpha-to-tilde-alpha}\ref{item:dom-f}, there exists a unique element of $\supp'(f)$ with source $\rho(\chi_i)$, which we will denote by $\tau_i$.
Let $\gamma_i:=\qsec(\dot{\gamma_i})$.
It follows from the 
fact that $f, f_i$ are
supported in bisections
 that $\tau_i\inv \gamma_i$
must be in $\agpd $.
In other words, $\dot{\tau}_i=\dot{\gamma}_i$.

    We have already shown that $\phi_{\chi_{i}}\to\phi_{\chi}$, so we know that $s(\tau_{i})=\rho(\chi_{i})\to \rho(\chi)=s(\dot{\gamma})$.
    Since $\tau_{i}$ and $\gamma:=\qsec (\dot{\gamma})$ are elements of the bisection $\supp'(f)$ where $s$ restricts to a homeomorphism, we conclude that $\tau_{i}\to \gamma.$ In particular 
    $\dot{\gamma}_{i} = \dot{\tau}_{i} \to \dot{\gamma}.$

    It remains to show that $\lambda_{i}\to \lambda.$         Since $\qsec $ is continuous, we know that $\dot{\gamma}_i\to\dot{\gamma}$ implies $\gamma_{i}=\qsec (\dot{\gamma}_i)\to \qsec (\dot{\gamma})=\gamma$. As $\supp'(f)$ is an open neighborhood of $\gamma$, there exists $i_{1}$ so that $\gamma_{i}\in \supp'(f)$ for $i\geq i_{1}$; to be precise, $\gamma_{i}=\tau_{i}$.  As $\gamma_i \in \supp'(f) \cap \supp'(f_i)$ for such $i$, the fact that  $f, f_i \in N$ enables us to invoke Corollary~\ref{cor:explicit-lambda}.  We conclude that
    \[
        \kappa_i
        =\frac{\overline{f(\gamma_i )} }{\abs{f(\gamma_i )}}\frac{ \lambda_i f_{\dot{\gamma}_i}(\gamma_i)}{\abs{ f_{\dot{\gamma}_i}(\gamma_i)}}
        =
        \frac{\overline{f(\gamma_i )} }{\abs{f(\gamma_i )}}\lambda_i
        .
    \]
Since $\gamma_{i}=\tau_{i}\to \gamma$ and the continuous function $f$ satisfies $f(\gamma)=f_{\dot{\gamma}}(\qsec (\dot{\gamma}))>0$, it follows that $\frac{f(\gamma_{i})}{\abs{f(\gamma_{i})}}$ converges to $1$. Consequently, $\overline{\kappa_{i}}\lambda_{i}\to 1$. As $\kappa_{i}$ converges to $\lambda$ since $\kappa_i\in V$ for all $i\geq j_{\dom(f),V}$, we have that $\lambda_{i}= \kappa_{i}(\overline{\kappa_{i}}\lambda_{i})$ converges to $\lambda$ also. 
    This finishes our proof.
\end{proof}

\begin{rmk}
    Note that, according to Theorem~\ref{thm:psi}, any choice of section $\qsec $ gives a $\mbb{T}$-equivariant isomorphism $\psi_{\qsec }$ of the {\em discrete} groupoids $( (\qgpd  \ltimes \specB  ) \times_{C^{\qsec }} \mbb{T})^{d}$ and $\Sigma_{(A,B)}^{d}$. In particular, we can let $\mc{E}_{\qsec }$ be the groupoid $ (\qgpd  \ltimes \specB  ) \times_{C^{\qsec }}\mbb{T}$ equipped with the initial topology with respect to $\psi_{\qsec }$, i.e.\ the coarsest topology that makes $\psi_{\qsec }$ continuous.
    This procedure makes $\psi_{\qsec }$ a $\mbb{T}$-equivariant isomorphism between the {\em topological} groupoids $\mc{E}_{\qsec }$ and $\Sigma_{(A,B)}$. If $\qsec $ is continuous, Theorem~\ref{thm:psi} establishes that $\mc{E}_{\qsec }$ is exactly $ (\qgpd  \ltimes \specB  ) \times_{C^{\qsec }} \mbb{T}$ with its usual topology.
\end{rmk}

\pagebreak[3]\section{Appendix}

\begin{lemma}\label{lem:main-cocycle-claim-for-C}
    Equation~\eqref{eq:main-cocycle-claim-for-C} holds.
\end{lemma}

\begin{proof}
Let $x:=\gamma_2\inv \gamma_1\inv \gamma_{12}, y:= \gamma_3\inv \gamma_2\inv \gamma_{23}, z:= \gamma_3\inv\gamma_2\inv \gamma_1\inv \gamma_{123}$. With these identifications, Equation \eqref{eq:main-cocycle-claim-for-C} becomes
\allowdisplaybreaks
\begin{align*}
 &
	c( z\inv y , y\inv )\,
	c( z\inv \gamma_{3}\inv \gamma_{2}\inv \gamma_{1}\inv, \gamma_{1})\,
	c( z\inv y y\inv \gamma_{3}\inv \gamma_{2}\inv, \gamma_{2} \gamma_{3} y )
	\\
	&\overline{c( y\inv \gamma_{3}\inv \gamma_{2}\inv, \gamma_{2} \gamma_{3} y )}\,
	c( y\inv \gamma_{3}\inv \gamma_{2}\inv, \gamma_{2})\,
	c( y\inv \gamma_{3}\inv, \gamma_{3})
	\\
	\overset{!}{=}\,
	&
	c( z\inv \gamma_{3}\inv x \gamma_{3}, \gamma_{3}\inv x\inv \gamma_{3})\,
	c( z\inv \gamma_{3}\inv \gamma_{2}\inv \gamma_{1}\inv, \gamma_{1} \gamma_{2} x )\,
	c( z\inv \gamma_{3}\inv x \gamma_{3} \gamma_{3}\inv, \gamma_{3})
	\\
	&\overline{c( \gamma_{3}\inv, \gamma_{3})}\,
	\overline{c( x\inv \gamma_{2}\inv \gamma_{1}\inv, \gamma_{1} \gamma_{2} x )}\,
	c( \gamma_{3}\inv, x\inv \gamma_{2}\inv \gamma_{1}\inv)
	\\
	&
	c( \gamma_{3}\inv x\inv \gamma_{2}\inv \gamma_{1}\inv, \gamma_{1})\,
	c( \gamma_{3}\inv x\inv \gamma_{3} \gamma_{3}\inv \gamma_{2}\inv, \gamma_{2})\,
	c( \gamma_{3}\inv x\inv \gamma_{3} \gamma_{3}\inv, \gamma_{3})
	\\\iff &\qquad \text{(bring conjugates to the other side)}\\
	&
 	c( z\inv y , y\inv )\,
 	c( z\inv \gamma_{3}\inv \gamma_{2}\inv \gamma_{1}\inv, \gamma_{1})\,
 	c( z\inv \gamma_{3}\inv \gamma_{2}\inv, \gamma_{2} \gamma_{3} y )
 	\\
 	&
 	c( x\inv \gamma_{2}\inv \gamma_{1}\inv, \gamma_{1} \gamma_{2} x )\,
 	\uwave{
 		c( y\inv \gamma_{3}\inv \gamma_{2}\inv, \gamma_{2})\,
 		c( y\inv \gamma_{3}\inv, \gamma_{3})\,
 	}
 	c( \gamma_{3}\inv, \gamma_{3})
 	\\
 	&\qquad \text{(use cocycle condition, Equation~\eqref{eq:cocycle}, on underlined factors)}
 	\\
	=&
 	c( z\inv y , y\inv )\,
 	c( z\inv \gamma_{3}\inv \gamma_{2}\inv \gamma_{1}\inv, \gamma_{1})\,
 	c( z\inv \gamma_{3}\inv \gamma_{2}\inv, \gamma_{2} \gamma_{3} y )
 	\\
 	&
 	c( x\inv \gamma_{2}\inv \gamma_{1}\inv, \gamma_{1} \gamma_{2} x )\,
 	{
 		c(y\inv \gamma_{3}\inv \gamma_{2}\inv,\gamma_{2}\gamma_{3})\,
 		c(\gamma_{2},\gamma_{3})
 	}
 	c( \gamma_{3}\inv, \gamma_{3})
 	\\
 	\overset{!}{=}\,
 	&
 	c( z\inv \gamma_{3}\inv x \gamma_{3}, \gamma_{3}\inv x\inv \gamma_{3})\,
 	c( z\inv \gamma_{3}\inv \gamma_{2}\inv \gamma_{1}\inv, \gamma_{1} \gamma_{2} x )\,
 	c( z\inv \gamma_{3}\inv x  , \gamma_{3})
 	\\
 	&
 	c( y\inv \gamma_{3}\inv \gamma_{2}\inv, \gamma_{2} \gamma_{3} y )\,
 	c( \gamma_{3}\inv, x\inv \gamma_{2}\inv \gamma_{1}\inv)
 	\\
	&
	c( \gamma_{3}\inv x\inv \gamma_{2}\inv \gamma_{1}\inv, \gamma_{1})\,
	\uwave{c( \gamma_{3}\inv x\inv \gamma_{2}\inv, \gamma_{2})\,
	c( \gamma_{3}\inv x\inv , \gamma_{3})} 
	\\
 	&\qquad \text{(use cocycle condition  on underlined factors)}
 	\\
	=&
 	c( z\inv \gamma_{3}\inv x \gamma_{3}, \gamma_{3}\inv x\inv \gamma_{3})\,
 	c( z\inv \gamma_{3}\inv \gamma_{2}\inv \gamma_{1}\inv, \gamma_{1} \gamma_{2} x )\,
 	c( z\inv \gamma_{3}\inv x  , \gamma_{3})
 	\\
 	&
 	c( y\inv \gamma_{3}\inv \gamma_{2}\inv, \gamma_{2} \gamma_{3} y )\,
 	c( \gamma_{3}\inv, x\inv \gamma_{2}\inv \gamma_{1}\inv)
 	\\
	&
	\uwave{
		c( \gamma_{3}\inv x\inv \gamma_{2}\inv \gamma_{1}\inv, \gamma_{1})\,
		{c(\gamma_{3}\inv x\inv \gamma_{2}\inv,\gamma_{2}\gamma_{3})}
	} {c(\gamma_2,\gamma_3)}
	\\
 	&\qquad \text{(use cocycle condition  on underlined factors)}
 	\\
	=&
 	c( z\inv \gamma_{3}\inv x \gamma_{3}, \gamma_{3}\inv x\inv \gamma_{3})\,
 	c( z\inv \gamma_{3}\inv \gamma_{2}\inv \gamma_{1}\inv, \gamma_{1} \gamma_{2} x )\,
 	c( z\inv \gamma_{3}\inv x  , \gamma_{3})
 	\\
 	&
 	c( y\inv \gamma_{3}\inv \gamma_{2}\inv, \gamma_{2} \gamma_{3} y )\,
	\uwave{c( \gamma_{3}\inv, x\inv \gamma_{2}\inv \gamma_{1}\inv)}
	\\
	&
	\uwave{c(\gamma_{3}\inv x\inv \gamma_{2}\inv \gamma_{1}\inv, \gamma_{1} \gamma_{2}\gamma_{3})\,}
		c(\gamma_{1}, \gamma_{2}\gamma_{3})
	c(\gamma_{2},\gamma_{3})	
	\\
	&\qquad \text{(use cocycle condition  on underlined factors)}
 	\\
 	=&
 	c( z\inv \gamma_{3}\inv x \gamma_{3}, \gamma_{3}\inv x\inv \gamma_{3})\,
 	c( z\inv \gamma_{3}\inv \gamma_{2}\inv \gamma_{1}\inv, \gamma_{1} \gamma_{2} x )\,
 	c( z\inv \gamma_{3}\inv x  , \gamma_{3})
 	\\
 	&
 	c( y\inv \gamma_{3}\inv \gamma_{2}\inv, \gamma_{2} \gamma_{3} y )\,
	{c(\gamma_{3}\inv,x\inv \gamma_{3})}
	\\
	&
	{c(x\inv \gamma_{2}\inv \gamma_{1}\inv,\gamma_{1} \gamma_{2}\gamma_{3})\,}
	c(\gamma_{1}, \gamma_{2}\gamma_{3})
	c(\gamma_{2},\gamma_{3})
\end{align*}
\begin{align*}
\iff&\qquad\text{(cancel $c(\gamma_{2},\gamma_{3})$ on both sides)}\\
	&
 	c( z\inv y , y\inv )\,
 	c( z\inv \gamma_{3}\inv \gamma_{2}\inv \gamma_{1}\inv, \gamma_{1})\,
 	c( z\inv \gamma_{3}\inv \gamma_{2}\inv, \gamma_{2} \gamma_{3} y )
 	\\
 	&
 	c( x\inv \gamma_{2}\inv \gamma_{1}\inv, \gamma_{1} \gamma_{2} x )\,
	c(y\inv \gamma_{3}\inv \gamma_{2}\inv,\gamma_{2}\gamma_{3})\,
  	c( \gamma_{3}\inv, \gamma_{3})
 	\\
 	\overset{!}{=}\,
 	&
 	\uwave{c( z\inv \gamma_{3}\inv x \gamma_{3}, \gamma_{3}\inv x\inv \gamma_{3})\,}
 	c( z\inv \gamma_{3}\inv \gamma_{2}\inv \gamma_{1}\inv, \gamma_{1} \gamma_{2} x )\,
 	c( z\inv \gamma_{3}\inv x  , \gamma_{3})
 	\\
 	&
 	c( y\inv \gamma_{3}\inv \gamma_{2}\inv, \gamma_{2} \gamma_{3} y )\,
	\uwave{c(\gamma_{3}\inv,x\inv \gamma_{3})}
	\\
	&
	c(x\inv \gamma_{2}\inv \gamma_{1}\inv,\gamma_{1} \gamma_{2}\gamma_{3})\,
	c(\gamma_{1}, \gamma_{2}\gamma_{3})\,
	\\
	&\qquad \text{(use cocycle condition  on underlined factors)}
 	\\
 	=
 	&
 	\uwave{c(z\inv \gamma_{3}\inv x \gamma_{3},\gamma_{3}\inv)\,}
 	c( z\inv \gamma_{3}\inv \gamma_{2}\inv \gamma_{1}\inv, \gamma_{1} \gamma_{2} x )\,
 	\uwave{c( z\inv \gamma_{3}\inv x  , \gamma_{3})}
 	\\
 	&
 	c( y\inv \gamma_{3}\inv \gamma_{2}\inv, \gamma_{2} \gamma_{3} y )\,
	{c(z\inv \gamma_{3}\inv x ,x\inv \gamma_{3})}
	\\
	&
	c(x\inv \gamma_{2}\inv \gamma_{1}\inv,\gamma_{1} \gamma_{2}\gamma_{3})\,
	c(\gamma_{1}, \gamma_{2}\gamma_{3})\,
	\\&\qquad \text{(use cocycle condition  on underlined factors)}
 	\\
 	=
 	&
 	c( z\inv \gamma_{3}\inv \gamma_{2}\inv \gamma_{1}\inv, \gamma_{1} \gamma_{2} x )\,
	{ c(\gamma_{3}\inv,\gamma_{3})}
 	\\
 	&
 	c( y\inv \gamma_{3}\inv \gamma_{2}\inv, \gamma_{2} \gamma_{3} y )\,
	c(z\inv \gamma_{3}\inv x ,x\inv \gamma_{3})
	\\
	&
	c(x\inv \gamma_{2}\inv \gamma_{1}\inv,\gamma_{1} \gamma_{2}\gamma_{3})\,
	c(\gamma_{1}, \gamma_{2}\gamma_{3})
	\\\iff&\qquad \text{(cancel $c(\gamma_3\inv,\gamma_3)$ on both sides)}
 	\\\
	&
 	c( z\inv y , y\inv )\,
 	c( z\inv \gamma_{3}\inv \gamma_{2}\inv \gamma_{1}\inv, \gamma_{1})\,
 	\uwave{c( z\inv \gamma_{3}\inv \gamma_{2}\inv, \gamma_{2} \gamma_{3} y )}
 	\\
 	&
 	c( x\inv \gamma_{2}\inv \gamma_{1}\inv, \gamma_{1} \gamma_{2} x )\,
	c(y\inv \gamma_{3}\inv \gamma_{2}\inv,\gamma_{2}\gamma_{3})\,
	\\&\qquad \text{(use cocycle condition  on underlined factor)}
 	\\
	=&
	\uwave{	c( z\inv y , y\inv )}\,
 	c( z\inv \gamma_{3}\inv \gamma_{2}\inv \gamma_{1}\inv, \gamma_{1})\,
 	{\uwave{c(z\inv,y)}\,
 	c(z\inv \gamma_{3}\inv \gamma_{2}\inv,\gamma_{2} \gamma_{3})\,
 	\overline{c(\gamma_{2} \gamma_{3},y)}
 	}
 	\\
 	&
 	c( x\inv \gamma_{2}\inv \gamma_{1}\inv, \gamma_{1} \gamma_{2} x )\,
	c(y\inv \gamma_{3}\inv \gamma_{2}\inv,\gamma_{2}\gamma_{3})\,
 	\\&\qquad \text{(use cocycle condition on underlined factors)}
 	\\
 		=&
		{\uwave{c(y,y\inv)}\,}
 	 	c( z\inv \gamma_{3}\inv \gamma_{2}\inv \gamma_{1}\inv, \gamma_{1})\,
 	 	c(z\inv \gamma_{3}\inv \gamma_{2}\inv,\gamma_{2} \gamma_{3})\,
 	 	\uwave{\overline{c(\gamma_{2} \gamma_{3},y)}}
 	 	\\
 	 	&
 	 	c( x\inv \gamma_{2}\inv \gamma_{1}\inv, \gamma_{1} \gamma_{2} x )\,
 		\uwave{c(y\inv \gamma_{3}\inv \gamma_{2}\inv,\gamma_{2}\gamma_{3})}
 	 	\\
	\overset{!}{=}&\,
 	c( z\inv \gamma_{3}\inv \gamma_{2}\inv \gamma_{1}\inv, \gamma_{1} \gamma_{2} x )
 	\uwave{c( y\inv \gamma_{3}\inv \gamma_{2}\inv, \gamma_{2} \gamma_{3} y )}\,
 	\,
 	\\
 	&
	c(z\inv \gamma_{3}\inv x ,x\inv \gamma_{3})
	c(x\inv \gamma_{2}\inv \gamma_{1}\inv,\gamma_{1} \gamma_{2}\gamma_{3})\,
	c(\gamma_{1}, \gamma_{2}\gamma_{3})
\end{align*}
\begin{align*}
	\iff&\qquad \text{(use cocycle condition  to cancel underlined factors)}
 	\\
		&\uwave{c( z\inv \gamma_{3}\inv \gamma_{2}\inv \gamma_{1}\inv, \gamma_{1})\,
 	 	c(z\inv \gamma_{3}\inv \gamma_{2}\inv,\gamma_{2} \gamma_{3})\,}
 	 	c( x\inv \gamma_{2}\inv \gamma_{1}\inv, \gamma_{1} \gamma_{2} x )\,
 	\\&\qquad \text{(use cocycle condition  on underlined factors)}\\
 	=&
	{c(z\inv \gamma_{3}\inv \gamma_{2}\inv \gamma_{1}\inv,\gamma_{1}\gamma_{2} \gamma_{3})\,
 		{c(\gamma_{1},\gamma_{2} \gamma_{3})}
 		\,}
 	c( x\inv \gamma_{2}\inv \gamma_{1}\inv, \gamma_{1} \gamma_{2} x )\,
 	\\
	\overset{!}{=}&\,
 	c( z\inv \gamma_{3}\inv \gamma_{2}\inv \gamma_{1}\inv, \gamma_{1} \gamma_{2} x )\,
	c(z\inv \gamma_{3}\inv x ,x\inv \gamma_{3})
	c(x\inv \gamma_{2}\inv \gamma_{1}\inv,\gamma_{1} \gamma_{2}\gamma_{3})\,
	{c(\gamma_{1}, \gamma_{2}\gamma_{3})}
	\\
	\\\iff&\qquad \text{(cancel $c(\gamma_{1}, \gamma_{2}\gamma_{3})$ on both sides)}\\
	&
	{c(z\inv \gamma_{3}\inv \gamma_{2}\inv \gamma_{1}\inv,\gamma_{1}\gamma_{2} \gamma_{3})\,}
 	c( x\inv \gamma_{2}\inv \gamma_{1}\inv, \gamma_{1} \gamma_{2} x )\,
	\\
		\overset{!}{=}&\,
	 	c( z\inv \gamma_{3}\inv \gamma_{2}\inv \gamma_{1}\inv, \gamma_{1} \gamma_{2} x )
	  	\\
	 	&
		\uwave{c(z\inv \gamma_{3}\inv x ,x\inv \gamma_{3})
		c(x\inv \gamma_{2}\inv \gamma_{1}\inv,\gamma_{1} \gamma_{2}\gamma_{3})\,}
	\\&\qquad \text{(use cocycle condition  on underlined factors)}\\
	=&\,
 	c( z\inv \gamma_{3}\inv \gamma_{2}\inv \gamma_{1}\inv, \gamma_{1} \gamma_{2} x )
  	\\
 	&
	{{c(z\inv \gamma_{3}\inv \gamma_{2}\inv \gamma_{1}\inv,\gamma_{1} \gamma_{2}\gamma_{3})\,}
		c(z\inv \gamma_{3}\inv x,x\inv \gamma_{2}\inv \gamma_{1}\inv)
	\,}
		\\\iff&\qquad \text{(cancel $c(z\inv \gamma_{3}\inv \gamma_{2}\inv \gamma_{1}\inv,\gamma_{1} \gamma_{2}\gamma_{3})$ on both sides)}\\
		&
	 	c( x\inv \gamma_{2}\inv \gamma_{1}\inv, \gamma_{1} \gamma_{2} x )
	 	\\&\qquad \text{(use cocycle condition )}\\
	 	=&
	 	c(x\inv, x)\,
	 	c(\gamma_{2}\inv \gamma_{1}\inv,\gamma_{1} \gamma_{2} x)\,
	 	\overline{c(x\inv,\gamma_{2}\inv \gamma_{1}\inv)}
		\\
	\overset{!}{=}&\,
	 	c( z\inv \gamma_{3}\inv \gamma_{2}\inv \gamma_{1}\inv, \gamma_{1} \gamma_{2} x )\,
		c(z\inv \gamma_{3}\inv x,x\inv \gamma_{2}\inv \gamma_{1}\inv)
	\\\iff&\qquad \text{(bring conjugate to other side)}\\
	&
 	c(x\inv, x)\,
 	{c(\gamma_{2}\inv \gamma_{1}\inv,\gamma_{1} \gamma_{2} x)\,}
	\\
\overset{!}{=}&\,
 	\uwave{c( z\inv \gamma_{3}\inv \gamma_{2}\inv \gamma_{1}\inv, \gamma_{1} \gamma_{2} x )\,}
	\dotuline{c(z\inv \gamma_{3}\inv x,x\inv \gamma_{2}\inv \gamma_{1}\inv)\,
	c(x\inv,\gamma_{2}\inv \gamma_{1}\inv)}
	\\&\qquad \text{(use cocycle condition  on underlined factors separately)}\\
	=&\,
 	{c(z\inv \gamma_{3}\inv,x)\,
 		{c(\gamma_{2}\inv \gamma_{1}\inv,\gamma_{1} \gamma_{2} x)\,}
 		\overline{c(z\inv \gamma_{3}\inv, \gamma_{2}\inv \gamma_{1}\inv)}\,}
 	\\&
	{c(z\inv \gamma_{3}\inv,\gamma_{2}\inv \gamma_{1}\inv)\,
		c(z\inv \gamma_{3}\inv x,x\inv)
		}
	\\\iff&\qquad \text{(cancel $c(\gamma_{2}\inv \gamma_{1}\inv,\gamma_{1} \gamma_{2} x)$ on both sides, bring conjugate to other side)}\\
	&
 	c(x\inv, x)\,c(z\inv \gamma_{3}\inv, \gamma_{2}\inv \gamma_{1}\inv)
	\\
\overset{!}{=}&\,
 		\uwave{c(z\inv \gamma_{3}\inv,x)\,}
		c(z\inv \gamma_{3}\inv,\gamma_{2}\inv \gamma_{1}\inv)\,
		\uwave{c(z\inv \gamma_{3}\inv x,x\inv)}
		\\&\qquad \text{(use cocycle condition  on underlined factors to arrive at a true statement)}\\
		=&\,
		c(z\inv \gamma_{3}\inv,\gamma_{2}\inv \gamma_{1}\inv)\,
		{c(x,x\inv)}.
\end{align*}
Note that, by  \cite[Lemma 2.1]{DGNRW:Cartan}, $c(x,x\inv)=c(x\inv,x)$ for any element of $\gpd $, so this Lemma does not rely on $c$ being symmetric on $\agpd $.
\end{proof}

\bibliographystyle{amsalpha} 

\newcommand{\etalchar}[1]{$^{#1}$}
\providecommand{\bysame}{\leavevmode\hbox to3em{\hrulefill}\thinspace}
\providecommand{\MR}{\relax\ifhmode\unskip\space\fi MR }
\providecommand{\MRhref}[2]{%
  \href{http://www.ams.org/mathscinet-getitem?mr=#1}{#2}
}
\providecommand{\href}[2]{#2}

\end{document}